\documentclass[onefignum,onetabnum]{siamart190516}

\usepackage{stmaryrd}

\usepackage{xparse}

\newcommand{\avg}[1]{\left\{\hspace{-2.1mm}\left\{ #1 \right\}\hspace{-2.1mm}\right\}}
\newcommand{\jp}[1]{\left[\!\!\left[ #1 \right]\!\!\right]}
\newcommand{\enorm}[1]{|\hspace{-0.3mm}|\hspace{-0.3mm}|#1|\hspace{-0.3mm}|\hspace{-0.3mm}|}
\newcommand{\dist}{\mathrm{dist}}

\NewDocumentCommand{\dgal}{sO{}m}{%
  \IfBooleanTF{#1}
    {\dgalext{#3}}
    {\dgalx[#2]{#3}}%
}

\NewDocumentCommand{\dgalext}{m}{%
  \sbox0{%
    \mathsurround=0pt 
    $\left\{\vphantom{#1}\right.\kern-\nulldelimiterspace$%
  }%
  \sbox2{\{}%
  \ifdim\ht0=\ht2
    \{\kern-.45\wd2 \{#1\}\kern-.45\wd2 \}%
  \else
  \fi
}

\NewDocumentCommand{\dgalx}{om}{%
  \sbox0{\mathsurround=0pt$#1\{$}%
  \sbox2{\{}%
  \ifdim\ht0=\ht2
    \{\kern-.45\wd2 \{#2\}\kern-.45\wd2 \}%
  \else
    \mathopen{#1\{\kern-.5\wd0 #1\{}
    #2
    \mathclose{#1\}\kern-.5\wd0 #1\}}
  \fi
}

\usepackage{lipsum}
\usepackage{amsfonts}
\usepackage{graphicx}
\usepackage{epstopdf}
\usepackage{algorithmic}

\usepackage{animate}
\usepackage{amsmath}
\usepackage{amssymb}
\usepackage{graphics}
\usepackage{graphicx}
\usepackage{footnote}
\usepackage{textcomp}
\usepackage{mathrsfs}
\usepackage{epstopdf}
\usepackage{array}
\usepackage[maxfloats=99]{morefloats}
\usepackage{url}
\usepackage{cases}
\usepackage{mathscinet}
\usepackage[normalem]{ulem}
\usepackage{algorithmic, algorithm}
\usepackage{color}
\usepackage{cite}

\ifpdf
  \DeclareGraphicsExtensions{.eps,.pdf,.png,.jpg}
\else
  \DeclareGraphicsExtensions{.eps}
\fi

\newsiamremark{remark}{Remark}
\newsiamremark{hypothesis}{Hypothesis}
\crefname{hypothesis}{Hypothesis}{Hypotheses}
\newsiamthm{claim}{Claim}

\headers{Superconvergent postprocessing of C0IP method}{Ying Cai, Hailong  Guo, and Zhimin Zhang}

\title{Superconvergent postprocessing of $C^0$ interior penalty method}

\author{Ying Cai\thanks{Beijing Computational Science Research Center, Beijing 100193, China (\email{ycai@csrc.ac.cn}).}
\and Hailong Guo\thanks{School of Mathematics and Statistics, The University of Melbourne, Parkville, VIC 3010, Australia (\email{hailong.guo@unimelb.edu.au}).}
\and Zhimin Zhang\thanks{Department of Mathematics, Wayne State University, Detroit, MI 48202, USA (\email{ag7761@wayne.edu}).}
}

\usepackage{amsopn}

\ifpdf
\hypersetup{
  pdftitle={Superconvergence for C0IP},
  pdfauthor={Hailong Guo}
}
\fi

\usepackage{subcaption}

\usepackage{multirow}

\graphicspath{{fig/}{fig/mesh/}{fig/adaptive}}

\usepackage{stmaryrd}

\begin{document}

\maketitle

\begin{abstract}
This paper focuses on the superconvergence analysis of the Hessian recovery technique for the $C^0$ Interior Penalty Method (C0IP) in solving the biharmonic equation. We establish interior error estimates for C0IP method that serve as the superconvergent analysis tool. Using the argument of superconvergence by difference quotient, we prove superconvergent results of the recovered Hessian matrix on translation-invariant meshes. The Hessian recovery technique enables us to construct an asymptotically exact {\it a posteriori} error estimator for the C0IP method. Numerical experiments are provided to support our theoretical results.

\end{abstract}



\section{Introduction}
\label{sec:int}

%

The biharmonic equation arises from many important applications of science and engineering areas, such as thin plate theory and fluid dynamics. Over the past century, there has been a dramatic increase in the developing finite element methods for the biharmonic equation (or general fourth-order elliptic problems). For the fourth-order elliptic partial differential equations (PDEs), the conforming finite element space is $C^1$. Famous examples include the  Argyris element and the Bell element \cite{Ci2002, BS2008}.  However, constructing such $C^1$  elements is rather complicated, and their implementation is far from straightforward.  To reduce computational costs, nonconforming elements, such as the Morley element \cite{Mo1968, LL1975, WangXu2006}, were proposed, decreasing the requirement of continuity by only necessitating weak continuity.  An alternative approach is to use mixed finite element methods\cite{CiRa1974, John1973}, which rewrite fourth-order PDEs into system of lower-order PDEs. Some other finite element methods include the discontinuous Galerkin method\cite{Baker1977} and the recovery based finite element methods\cite{CGZZ2017, GZZ2018}.

The idea of using a continuous finite element method to discretize fourth-order elliptic PDEs can be traced back to the 1970s. In their seminal work \cite{BaZl1973}, Babu\v{s}ka and Zl\'{a}mal proposed a C0IP method using cubic Hermite elements. However, the method is inconsistent, which leads to a suboptimal convergence rate. Eagles et al. proposed a consistent C0IP method in \cite{EGHLMT2002}. This method employs the Lagrange finite element space and weakly enforces continuity through stabilization. Its optimal convergence in the energy norm and $L^2$ norm was analyzed in \cite{BS2005}. Due to its flexibility and ease of implementation, it has become one of the most popular discretization methods for fourth-order elliptic PDEs.

Adaptive computation plays a fundamental role in modern scientific and engineering computing. One of the key ingredients of adaptive methods is {\it a posteriori} error estimators. Existing error estimators can be roughly divided into two categories: residual-type \cite{AO2000, BS2001, VR2013} and recovery-type. However, most studies in the field of {\it a posteriori} estimates for the C0IP method have only focused on residual-type error estimators \cite{BGSu2010, CNRS2023}.
A realistic approach to constructing recovery-type {\it a posterior} error estimators is to perform a superconvergent postprocess on finite element solutions. For second-order elliptic PDEs, the superconvergence of gradient recovery techniques and their applications in {\it a posteriori} error estimates have reached a mature stage, as seen in \cite{NZ2004, ZN2005, ZZ1992a}.  When dealing with fourth-order elliptic PDEs, a natural choice is Hessian recovery \cite{AgVa2002, PABG2011, GZZ2017b}.

Our primary goals are to establish the superconvergence theory of Hessian recovery proposed in \cite{GZZ2017b} for the C0IP method and use it to construct a simple and asymptotically exact \textit{a posterior} error estimator for the biharmonic equation.
 There are two possible directions to demonstrate the superconvergence of Hessian recovery:
one involves using existing superconvergence results, and the other relies on the argument of superconvergence by difference quotient \cite{Wa1995, ZN2005} and interior error estimates for function values. However, numerical tests indicate no supercloseness for the C0IP method. It appears that the second approach is the only way to establish the superconvergence theory of Hessian recovery for the C0IP method. Regarding pointwise error analysis, Leykekhman \cite{Le2021} estimated the global and local maximum norms of second-order derivatives for the C0IP method. To the best of our knowledge, there are no interior error estimates in function value for the C0IP method.

There are four contributions in this paper. First, we establish interior error estimates for the C0IP method in energy and $L^2$ norms. For second-order elliptic problems, this aspect has been addressed by Nitsche and Schatz in \cite{NS1974}, and we will follow the methodology outlined in \cite{NS1974}. By proving some interior duality estimates, we demonstrate that the finite element error in the energy norm can be controlled with the best order of accuracy of the discretization space provided over the entire domain, coupled with an error in the negative Sobolev norm, which represents a weak norm. Similar estimates are also developed in the $L^2$ norm.
Second, based on the interior energy error estimates, we conduct interior maximum norm error estimates. Note that the discrete bilinear form may not be coercive locally; similar to the technique in \cite{SW1977}, we also need to consider the enhanced bilinear form associated with a Neumann problem. Consequently, we first carry out the estimation on the enhanced problems, and subsequently obtain the estimates on the original problems. The core idea of our analysis lies in the local estimate in the discrete $W_1^2$ norm.
Third, utilizing the interior estimates we have developed, we prove superconvergence results of the Hessian recovery for translation-invariant meshes in $L^2$ and maximum norms. To the best of our knowledge, these are the first superconvergence results for the C0IP method.
Last but not least, we propose a recovery-based {\it a posteriori} error estimator for the C0IP method. Compared to existing residual-based 
{\it a posteriori} error estimators \cite{BGSu2010, CNRS2023}, it offers several advantages, including asymptotical exactness, robustness, and ease of implementation.

The rest of the paper is organized as follows:
In Section \ref{sec:biharmonic}, we introduce the biharmonic equations and their $C^0$ interior penalty discretization. Section \ref{sec:hessianrecovery} provides a brief review of the gradient recovery procedure and introduces the Hessian recovery method. The interior estimates in energy, $L^2$, and $L^\infty$ norms are presented in Section \ref{sec:interiorestimate}. The core theorems in this section are Theorem \ref{thm:locinest} and Theorem \ref{thm:locinfty}, which serve as crucial tools for conducting our superconvergence analysis.
Building upon the results in Section \ref{sec:interiorestimate}, Section \ref{sec:error} establishes the superconvergent result for the recovered Hessian matrix. Finally, in Section \ref{sec:numer}, we present  some numerical examples that confirm the good performance of our method.

Throughout this paper, the symbol $C$, with or without subscripts, is adopted to denote unspecified positive constants, and we use notation $x \lesssim y$ to denote $x\le Cy$.
We also assume $\kappa$ is a fixed positive constant,
serving as a parameter to delineate the separation between the boundaries of  specific domains.

\section{$C^0$ interior penalty method for biharmonic equations}
\label{sec:biharmonic}
The purpose of this section is to provide some background materials. We shall begin by giving a brief introduction to our model equation. This will be followed by the $C^0$ interior penalty discretization of the model problem.
\subsection{Biharmonic equations}
\label{ssec:biharmonic}

Let $\Omega$ be a convex polygonal domain with a Lipschitz boundary in $\mathbb{R}^2$. For any subdomain $\mathcal{A}$ of $\Omega$, let $\mathbb{P}_m(\mathcal{A})$ be the space of polynomials of degree no more than $m$ over $\mathcal{A}$, and let $n_m$ be the dimension of $\mathbb{P}_m(\mathcal{A})$, where $n_m=\frac{1}{2}(m+1)(m+2)$. Following the notation used in \cite{AR1975,Ci2002,BS2008}, let $W^k_p(\mathcal A)$ denote the Sobolev space on $\mathcal{A}$ with the norm $\|\cdot\|_{k, p, \mathcal{A}}$ and the seminorm $|\cdot|_{k, p, \mathcal{A}}$. When $p=2$, the Sobolev space $W^k_p(\mathcal A)$ is abbreviated as $H^k(\mathcal{A})$, and the subscript $p$ in the (semi)norm can be omitted.

A 2-index $\alpha$ is a pair of nonnegative integers, $\alpha_i$, and the length of $\alpha$ is denoted as $|\alpha|= \alpha_1 + \alpha_2$. Given a 2-index $\alpha$, the following notation is used for derivatives:
\begin{equation}\label{equ:partial}
D^{\alpha}u := \frac{\partial^{|\alpha|}u}{\partial x^{\alpha_1} \partial y^{\alpha_2} } .
\end{equation}
Here, $D^{k}$ denotes the tensor of all partial derivatives of order $k$ for any given nonnegative integer $k$. Specifically, $D^2u$ represents the Hessian matrix of $u$. The Hessian operator $H$ is defined as:
\begin{equation}\label{equ:hessian}
H =
\begin{pmatrix}
\frac{\partial^2}{\partial x^2 } & \frac{\partial^2}{\partial x\partial y }\\
\frac{\partial^2}{\partial y\partial x } & \frac{\partial^2}{\partial y^2 }
\end{pmatrix}.
\end{equation}

We consider the following fourth-order elliptic equation
\begin{equation}\label{equ:model}
	\begin{split}
			\Delta^2u = f,  & \quad x\in \Omega, \\
	 u = \frac{\partial  u}{\partial n}  = 0, & \quad  x\in \partial\Omega,
	\end{split}
\end{equation}
where $f \in L^2(\Omega)$ and $n$ is the united out normal vector of $\Omega$.

The variational formulation of \eqref{equ:model} is to seek $u\in H^2_0(\Omega)$  such that
\begin{equation}\label{equ:var}
	B(u, v) = (f, v)\quad \forall v\in H_0^2(\Omega),
\end{equation}
where the bilinear form is defined as
\begin{equation}
	B(u, v) = \int_{\Omega} D^2u:D^2vdx
\end{equation}
with $M:N$ being the  Frobenius produce of two matrices $M$ and $N$.  It is not hard to deduce that
\begin{equation}
	\|u\|_{2, \Omega} \lesssim \|f\|_{-2, \Omega}.
\end{equation}

\subsection{$C^0$ interior penalty discretization}  For any $0 < h < \frac{1}{2}$, let $\mathcal{T}_h$ be  a shape regular triangulation of  $\Omega$  with mesh size at most $h$, i.e.
\begin{equation*}
	\overline{\Omega} = \bigcup_{T\in\mathcal{T}_h} T.
\end{equation*}

For any integer $k\geq 2$, define the continuous finite element space $S_h$ of order $k$ as
\begin{equation}
	S_h = \left\{ v\in C^0(\overline{\Omega}): v|_{T} \in \mathbb{P}_k(T), \, \forall \, T\in \mathcal{T}_h \right\} \subset H^1(\Omega).
\end{equation}
To handle the boundary condition, let $S_h^0 = S_h \cap H_0^1(\Omega)$ be the subspace satisfying the homogeneous Dirichlet boundary condition.
Denote the nodal points of $S_h$ by $\mathcal{N}_h$.  The Lagrange basis function of $S_h$ is denoted as $\{\phi_{z}: z\in \mathcal{N}_h\}$ which satisfies $\phi_z(z') = \delta_{zz'}$ for $z,z' \in \mathcal{N}_h$.  Let $I_h$ denote the interpolation operator from $C^0(\overline{\Omega})$ to $S_h$. For any $v\in C^0(\overline{\Omega})$, we have
\begin{equation*}
	I_hv = \sum_{z\in \mathcal{N}_h} v(z)\phi_z.
\end{equation*}
Furthermore, the following interpolation error estimate holds \cite{BS2001, Ci2002}
\begin{equation}
	\|v-I_hv\|_{j, p, T}  \lesssim h^{s-j}\|v\|_{s, p, T}, \quad T\in\mathcal{T}_h,
\end{equation}
where $1\le p \le \infty$ and $j \le s \le k+1$.

To discrete the fourth-order equation, we introduce the following local $H^2$ space as
\begin{equation}\label{equ:localh2}
	H^2(\Omega, \mathcal{T}_h) = \left\{v\in L^2(\Omega): v_T = v|_T \in H^2(T), \quad
	\forall T \in \mathcal{T}_h  \right\}.
\end{equation}
Denote the set of edges in $\mathcal{T}_h$ as $\mathcal{E}_h$.  For any interior edge $E\in \mathcal{E}_h$, let $T_+$ and $T_-$ be the two adjacent elements sharing $E$ as a common edge and $n_E$ be the  unit outer normal vector of $E$ pointing from $T_-$ to $T_+$. For a function $v \in H^2(\Omega, \mathcal{T}_h)$, define  the averaging and jump of $v$ as
\begin{equation}
	\avg{\frac{\partial^2 v}{\partial n^2} } = \frac12\left(\left. \frac{\partial^2 v_{T_+}}{\partial n^2}\right|_E + \left.\frac{\partial^2 v_{T_-}}{\partial n^2}\right|_E\right)
\text{  and  }
	 \jp{\frac{\partial v}{\partial n} } =\left. \frac{\partial v_{T_+}}{\partial n}\right|_E - \left.\frac{\partial v_{T_-}}{\partial n}\right|_E.
\end{equation}
 The above definitions are independent of the choice of $n_E$.
For any boundary edge $E\in \mathcal{E}_h$, let $n_E$ be the unit normal pointing outside $\Omega$.  In that case, averaging and jump of $v$ is defined as
\begin{equation}
	\avg{\frac{\partial^2 v}{\partial n^2} } =  - \frac{\partial^2 v}{\partial n_E^2}
\text{ and }
	 \jp{\frac{\partial v}{\partial n} } = - \frac{\partial v}{\partial n_E}.
\end{equation}

Define the discrete bilinear form $B_h(\cdot, \cdot)$ as
\begin{equation}
\begin{split}
	B_h(v, w) = &\sum_{T\in\mathcal{T}_h} \int_T D^2v:D^2wdx
	+ \sum_{E\in\mathcal{E}_h}\int_E \avg{\frac{\partial^2 v}{\partial n^2}}\jp{\frac{\partial w}{\partial n} } ds+\\
	&\sum_{E\in\mathcal{E}_h}\int_E \jp{\frac{\partial v}{\partial n}}
	\avg{\frac{\partial^2 w}{\partial n^2}}ds + \sum_{E\in\mathcal{E}_h}\frac{\gamma}{|E|}\int_E\jp{\frac{\partial v}{\partial n}}
	\jp{\frac{\partial w}{\partial n}} ds,
\end{split}
\end{equation}
where $\gamma$ is the penalty parameter.

The C0IP method for the model problem \eqref{equ:model} is to find $u_h\in S_h^0$ such that
\begin{equation}\label{equ:c0ip}
	B_h(u_h, v_h) = (f, v_h) \quad \forall v_h \in S_h^0.
\end{equation}
Provided that $u\in H^{2+\alpha}(\Omega)$ for some $\alpha>\frac12$, we have the following  Galerkin orthogonal equality \cite[eq. (4.16)]{BS2005}
\begin{equation}\label{equ:galorth}
	B_h(u_h-u, v_h) = 0 \quad \forall v_h \in S_h^0.
\end{equation}

Define the mesh-dependent norm $\interleave \cdot  \interleave_{2, h}$ as
\begin{equation}
\enorm v_{2, h}^2:=\sum_{T \in \mathcal{T}_h}\left\|D^2 v\right\|_{0,T}^2+\sum_{E \in \mathcal{E}_h}
\left(|E|\left\|\avg{\frac{\partial^2 v}{\partial n^2}}\right\|_{0,E}^2+|E|^{-1}\left\|\jp{ \frac{\partial v}{\partial n} }\right\|_{0,E}^2\right).
\end{equation}
It can be shown that the discrete bilinear form $B_h(\cdot, \cdot)$ satisfies the coercivity
\begin{equation}
	\enorm{v_h}_{2, h}^2 \lesssim B_h(v_h,v_h)
\end{equation}
and continuity
\begin{equation}
	B_h(v_h, w_h) \lesssim \enorm{v_h}_{2, h}\enorm{w_h}_{2, h}.
\end{equation}
The Lax-Milgram lemma\cite{Ci2002, BS2008} implies the discrete variational problem \eqref{equ:c0ip} admits a unique solution $u_h$. Furthermore, we have the following optimal error estimate
\begin{equation}\label{equ:interperror}
\enorm{ u - u_h }_{2, h}   \lesssim h^{k-1} \|u\|_{k+1, \Omega}.
\end{equation}
For the convenience of our subsequent interior error analysis, on a subset  $\mathcal A\subset\Omega$, we introduce the following mesh dependent semi-norms
\begin{align*}
&\enorm{v}_{2,1,h,\mathcal A}:=\sum_{T\in\mathcal T_h}|v|_{2,1,T\cap \mathcal A}+\sum_{E\in\mathcal E_h}\int_{E\cap \mathcal A}\left(|E|\left|\avg{\frac{\partial^2 v}{\partial n^2}}\right|+\left|\jp{\frac{\partial v}{\partial n}}\right|\right),\\
&\enorm{v}_{2,h,\mathcal A}^2:=\sum_{T \in \mathcal{T}_h}\left\|D^2 v\right\|_{0,T\cap \mathcal A}^2+\sum_{E \in \mathcal{E}_h}\left(|E|\left\|\avg{\frac{\partial^2 v}{\partial n^2}}\right\|_{0,E\cap \mathcal A}^2
+\frac1{|E|}\left\|\jp{\frac{\partial v}{\partial n}}\right\|_{0,e\cap \mathcal A}^2\right),\\
&\enorm{v}_{2,\infty,h,\mathcal A}:=\max_{T\in\mathcal T_h}|v|_{2,\infty,T\cap \mathcal A}+\max_{E\in\mathcal E_h}\left\|\avg{\frac{\partial^2 v}{\partial n^2}}\right\|_{0,\infty,E\cap \mathcal A}\\
&\hspace{8cm}+\max_{E\in\mathcal E_h}|E|^{-1}\left\|\jp{\frac{\partial v}{\partial n}}\right\|_{0,\infty,E\cap \mathcal A}.
\end{align*}

\section{Hessian recovery method}
\label{sec:hessianrecovery}
To obtain superconvergent results for the C0IP method, we need to do some postprocessing on the numerical solution $u_h$. In this paper, we focus on the postprocessing of the Hessian matrix of $u_h$, as introduced in \cite{GZZ2017b}. To prepare for presenting the Hessian recovery operator, we briefly describe the gradient recovery operator in \cite{ZN2005}.

Let $G_h: S_h \rightarrow S_h\times S_h $ be the polynomial preserving recovery (PPR) operator in \cite{ZN2005}. The gradient recovery operator $G_h$ can be realized in the following three steps: (1) building local patches of elements; (2) adopting local recovery procedures; (3) assembling the local recovered data into a global formulation.

 For each $z$ in $\mathcal N_h$, we can construct a local element of patch $\mathcal K_z$ satisfying the rank condition in the sense of the least-square problem \eqref{equ:lsq} being uniquely solvable \cite{NZ2004}.
 Select nodal point set $\mathcal K_z\cap\mathcal N_h$ as our sample point set and denote the indexes of it by $\mathcal I_z$.
 Using such sampling points,  fit a  polynomial $p_z$ of degree $(k+1)$ at the vertex $z$ in the following least-squares sense:
\begin{equation}\label{equ:lsq}
	p_z=\mathop{\mathrm{arg\,min}}\limits_{p \in \mathbb{P}_{k+1}(\mathcal K_z)} \sum_{j \in\mathcal I_z}|p(z_j)-u_h(z_j)|^2.
\end{equation}
 Once we define the recovered gradient for all nodal point in $\mathcal{N}_h$, the global formulation of $G_h$ is
\begin{equation}\label{equ:ppr}
	G_hu_h = \sum_{z\in\mathcal{N}_h}(G_hu_h)(z)\phi_z.
\end{equation}

\begin{remark}\label{rem:sample}
	It is worth to mention that the selection of sampling points in $\mathcal K_z$ is flexible as mentioned in \cite{ZN2005}.  For the theoretical analysis propose, we select the sampling point has the same type as $z$.
For example, if $z$ is diagonal edge center, we select all sampling points are diagonal edge centers.
\end{remark}

The Hessian recovery operator $H_h: S_h\rightarrow S_h^2\times S_h^2$ is obtained by applying $G_h$ twice\cite{GZZ2017b}.  Rewrite $G_hu_h$ as
\begin{equation}\label{equ:grcomp}
	G_h u_h=\begin{pmatrix}
G_h^x u_h \\
G_h^y u_h
\end{pmatrix}.
\end{equation}
Define $H_hu_h$ as
\begin{equation}\label{equ:hr}
	H_h u_h =(G_h(G_h^x u_h)\,\, G_h(G_h^y u_h))=
 \begin{pmatrix}
G_h^x(G_h^x u_h) & G_h^x(G_h^y u_h) \\
G_h^y(G_h^x u_h) & G_h^y(G_h^y u_h)
\end{pmatrix}.
\end{equation}

Given a subset $\mathcal{A}$ of $\Omega$, let $S_h(\mathcal{A})$  be the restriction of functions in $S_h$ to $\mathcal{A}$ and $S_h^{0}(\mathcal{A})$ be the set of those functions in $S_h(\mathcal{A})$ with compact support in the interior of $\mathcal{A}$\cite{Wa1995}. Suppose  $\Omega_0 \subset \subset \Omega_1 \subset \subset \Omega_2 \subset \subset \Omega$ is separated by $d \geq \kappa h$ and $\ell$ be a direction, i.e., a unit vector in $\mathbb{R}^2$. Let $\tau$ be a parameter, which will typically be a multiple of $h$. Let $T_\tau^{\ell}$ denote translation by $\tau$ in the direction $\ell$, i.e.,
\[
T_\tau^{\ell} v(z)=v(z+\tau \ell),
\]
and for an integer $\nu$,
\[
T_{\nu \tau}^{\ell} v(z)=v(z+\nu \tau \ell).
\]
The finite element space $S_h$ is referred as  translation invariant \cite{Wa1995} by $\tau$ in the direction $\ell$ if
\begin{equation}\label{equ:traninvariant}
T_{\nu \tau}^{\ell} v \in S_h^{0}(\Omega) \quad \forall v \in S_h^{0}(\Omega_1),	
\end{equation}
for some integer $\nu$ with $|\nu|<M$. Equivalently, $\mathcal{T}_h$ is called a translation invariant mesh. As demonstrated in \cite{GZZ2017b},  the five popular triangular mesh patterns: regular, chevron, criss-cross, union-Jack, and equilateral patterns, are translation invariant.

For the Hessian recovery operator $H_h$, \cite{GZZ2017b} established the following consistency property:
\begin{theorem}\label{thm:pp}
    Let $u\in W_{\infty}^{k+2}(\mathcal K_z)$ and $u_I=I_hu$; then
    \begin{equation*}
      \|H u-H_hu_I\|_{0, \infty, \mathcal K_z}\lesssim h^{k}|u|
	_{k+2, \infty,\mathcal K_z}.
    \end{equation*}
    If  $z$ is a node of translation invariant mesh and a mesh symmetric center of the involved nodes and $u\in W_{\infty}^{k+3}(\mathcal K_z)$, then
    \begin{equation*}
	|(H u-H_hu_I)(z)|\lesssim h^{k+1}|u|
	_{k+3, \infty,\mathcal K_z}.
    \end{equation*}
Moreover, if
$u\in W_{\infty}^{k+4}(\mathcal K_z)$  and  $k$  is an even number,  then
    \begin{equation*}
	|(H u-H_hu_I)(z)|\lesssim h^{k+2}|u|
	_{k+4, \infty,\mathcal K_z}.
    \end{equation*}
\end{theorem}

\section{Interior estimates for C0IP methods}
\label{sec:interiorestimate}
In this section, we perform the interior norm estimates for the C0IP method described in previous sections. In subsection \ref{sec:localenergyestimate}, we mainly consider the interior energy norm estimates. Based on the results in subsection \ref{sec:localenergyestimate}, we obtain the interior maximum norm estimates in subsection \ref{sec:localmaxestimate}, which primarily constitute our Theorem \ref{thm:locinfty}. Furthermore, Theorem \ref{thm:local} plays an extremely important role in the superconvergence analysis in Section \ref{sec:error}. The estimation techniques used in this section mainly rely on the methodology presented in \cite{NS1974} and \cite{SW1977,SW1995}.

In our analysis, it is necessary to frequently consider auxiliary problems on a local subdomain, and it is more appropriate to attach a Neumann-type boundary condition to the auxiliary problems considered on the local subdomain. In addition, observing that $B(\cdot,\cdot)$ and $B_h(\cdot,\cdot)$ are not coercive locally, for our convenience, we consider the following problem: Find $u\in H^2(\Omega)$ such that
\begin{equation}\label{eq:tildeB}
  \tilde B(u,v)=(f,v)\quad \forall v\in H^2(\Omega),
\end{equation}
where $\tilde B(u,v):=B(u,v)+K(u,v)$, with $K(u,v)=\int_\Omega uvdx$. The problem \eqref{eq:tildeB} is closed related to the plate equation with free boundary condition, see e.g, \cite{HL2002, MC2023}.
Correspondingly, the C0IP method for solving problem \eqref{eq:tildeB} is to seek $u_h\in S_h(\Omega)$ such that
 \begin{equation}\label{eq:dtildeB}
  \tilde B_h(u_h,v_h)=(f,v_h)\quad \forall v_h\in S_h(\Omega),
\end{equation}
where $\tilde B_h(u_h,v_h)=B_h(u_h,v_h)+K(u_h,v_h)$. The mesh-dependent (semi)norms associated to $\tilde B_h$ will be distinguished using superscript ``$\sim$''.

\subsection{Local error estimates in energy and $L^2$ norms}
\label{sec:localenergyestimate}
In this subsection, we consider interior error estimates in energy and $L^2$ norms.
Thought this subsection, we assume  $\Omega_0\subset\subset\Omega_1\subset\subset\Omega$,
 $u$ and $u_h\in S_h(\Omega_1)$ satisfy
\begin{equation}
  B_h(u-u_h,\chi)=0\quad\forall \chi\in S_h^0(\Omega_1).
\end{equation}
We denote $G_0\subset\subset G$, $\mathrm{dist}(G_0,\partial G)\gtrsim \kappa h$, are concentric spheres with $G\subset\subset\Omega_1$. To alleviate the notation, we set $e=u-u_h$. First, we consider the interior duality estimates.
\begin{lemma}\label{lem:dualone}
For any integer $s\geq -1$, let $\sigma=\min\{s+2, k-1\}$, we have
\begin{equation}\label{ineq:dualone}
  \|u-u_h\|_{-s,G_0}\lesssim h^\sigma\enorm{ u-\chi}_{2,h,G}+\|u-u_h\|_{-s-1,G}.
\end{equation}
\end{lemma}
\begin{proof}
Let $G_0\subset\subset G'\subset\subset G$ be concentric spheres with \[\mathrm{dist}(G_0,\partial G')=\mathrm{dist}(G',\partial G_1)\gtrsim \kappa h,\]
and let $\omega\in C_0^\infty(G')$ with $\omega\equiv1 $ on $G_0$.
Then, for $s\geq0$, we observe that
\begin{equation}\label{ineq:dualone1}
  \|e\|_{-s,G_0}\leq \|\omega e\|_{-s,G}=\sup_{f\in H_0^s(G)}\frac{(\omega e,f)}{\|f\|_{s,G}}.
\end{equation}
We employ the auxiliary problem
\[\Delta^2 w=f\quad\text{ in } G,\quad w=\frac{\partial w}{\partial n}=0\quad\text{ on }\partial G,\]
then according regularity estimate, see Lemma \ref{lem:rgl}, we have $w\in H^{s+4}(G)\cap H_0^2(G)$ and
\begin{equation}\label{ineq:reg1}
  \|w\|_{s+4,G}\lesssim \|f\|_{s,G}.
\end{equation}
Note that $\omega e$ is continuous, an application of integration by parts leads to
\[B_h(w, \omega e)=(\omega e,f).\]
Then, \eqref{ineq:dualone1} and \eqref{ineq:reg1} give that
\begin{equation}\label{ineq:dualone1-1}
\|e\|_{-s,G_0}\lesssim \sup_{w \in H^{s+4}(G)}\frac{B_h(w, \omega e)}{\|w\|_{s+4,G}}.
\end{equation}
Using the product rule, we derive
\[D^2(\omega e)=\omega D^2e+2\nabla\omega\otimes \nabla e+eD^2\omega\]
and
\[D^2(\omega w)=\omega D^2w+2\nabla\omega\otimes \nabla w+wD^2\omega.\]
Note that $[\![w]\!]=[\![\frac{\partial w}{\partial n}]\!]=0$, we have the identity
\begin{equation}\label{ineq:dualone2}
  \begin{split}
       & B_h(\omega e,w)-B_h(e, \omega w)\\
  =    &\sum_{T\in \mathcal T_h}\int_{T\cap G}(2\nabla \omega\otimes\nabla e+eD^2\omega):D^2w-(2\nabla \omega\otimes\nabla w+wD^2\omega):D^2e\\
  &-\sum_{E\in\mathcal E_h}\int_{E\cap G}\avg{2\frac{\partial \omega}{\partial n}\frac{\partial w}{\partial n}+w\frac{\partial^2\omega}{\partial n^2}}\jp{\frac{\partial e}{\partial n}}.
  \end{split}
\end{equation}
We deduce by Green's formula that
\begin{equation}\label{ineq:dualone3}
\sum_{T\in \mathcal T_h}\int_{T\cap G}2\nabla \omega\otimes\nabla e:D^2w=-\sum_{T\in \mathcal T_h}\int_{T\cap G}2e\mathrm{div}(D^2w\nabla \omega)
\end{equation}
and
\begin{equation}\label{ineq:dualone4}
\begin{split}
&  \sum_{T\in \mathcal T_h}\int_{T\cap G}(2\nabla \omega\otimes\nabla w+wD^2\omega):D^2e
  +\sum_{E\in\mathcal E_h}\int_{E\cap G}\avg{2\frac{\partial \omega}{\partial n}\frac{\partial w}{\partial n}+w\frac{\partial^2\omega}{\partial n^2}}\jp{\frac{\partial e}{\partial n}}\\
&= - \sum_{T\in \mathcal T_h}\int_{T\cap G}\mathrm{divdiv}(2\nabla \omega\otimes\nabla w+wD^2\omega)e.
\end{split}
\end{equation}
Therefore, from \eqref{ineq:dualone2}-\eqref{ineq:dualone4}, we claim that
\begin{equation}\label{ineq:dualone5}
\begin{split}
B_h(\omega e,w)&=B_h(e, \omega w)+\\
&\sum_{T\in \mathcal T_h}\int_{T\cap G}(-2\mathrm{div}(D^2w\nabla \omega)+D^2\omega:D^2w+\mathrm{divdiv}(2\nabla \omega\otimes\nabla w+wD^2\omega))e\\
:&=B_h(e, \omega w)+\Xi(w).
\end{split}
\end{equation}
We can derive  that $|\Xi(w)|\lesssim \|e\|_{-s-1,G}\|w\|_{s+4,G}$ from $\omega\in C_0^\infty(G')$. For $B_h(e, \omega w)$, according to the Galerkin orthogonality and the approximation theory, one obtains
\begin{equation}\label{ineq:dualone6}
\begin{split}
B_h(e, \omega w)&=B_h(e, \omega w-I_h(\omega w))\\&
\lesssim \enorm e_{2,h,G}\enorm{ \omega w-I_h(\omega w)}_{2,h,G}\lesssim h^\sigma\enorm e_{2,h,G}\|w\|_{s+4,G}.
\end{split}
\end{equation}
Then \eqref{ineq:dualone1-1} and \eqref{ineq:dualone6} give the lemma with $s\geq0$.

Similarly, for $s=-1$, note that $\omega e\in H_0^1(G)$, we have
\begin{equation}\label{ineq:dualtwo1}
  \|e\|_{1,G_0}\leq \|\omega e\|_{1,G}=\sup_{f\in H^{-1}(G)}\frac{(\omega e,f)}{\|f\|_{-1,G}}.
\end{equation}
Consider the dual problem
\[\Delta^2\tilde  w=f\quad\text{ in } G,\quad\tilde w=\frac{\partial\tilde w}{\partial n}=0\quad\text{ on }\partial G,\]
then $\tilde w\in H^3(G)\cap H_0^2(G)$ and $\|\tilde w\|_{3,G}\lesssim \|f\|_{-1,G}$. By the same manipulation as above, we obtain  \eqref{ineq:dualone5} holds with $\tilde w$ and
$|\Xi(\tilde w)|\lesssim \|e\|_{0,G}\|\tilde w\|_{3,G}$. On the other hand,
\begin{align*}
B_h(e, \omega\tilde w)&=B_h(e, \omega\tilde w-I_h(\omega\tilde w))\\&
\lesssim \enorm e_{2,h,G}\enorm{ \omega \tilde w-I_h(\omega\tilde w)}_{2,h,G}\lesssim
h^\sigma\enorm e_{2,h,G}\|\tilde w\|_{3,G},
\end{align*}
which combines \eqref{ineq:dualtwo1} completes the proof.
\end{proof}

\begin{lemma}\label{lem:localerror}
For any nonnegative integers $s$ and $q$, there holds that
\begin{equation}\label{ineq:localerror}
\|e\|_{1,G_0}\lesssim h\enorm e_{2,h,G}+\|e\|_{-s,q,G}.
\end{equation}
\end{lemma}
\begin{proof}
We choose integral $p$ such that  $W^{-s,q}(G)\hookrightarrow H^{-p}(G)$, let $G_0\subset\subset G_1\subset\subset\cdots\subset\subset G_{p}=G$ be concentric spheres
with the distance of any two adjacent spheres is greater than $\kappa h$, then we use Lemma \ref{lem:dualone} $p+2$ times to obtain
  \begin{equation}\label{ineq:localerror1}
    \begin{split}
      \|e\|_{1,G_0}&\lesssim h\enorm{e}_{2,h,G_1}+\|e\|_{0,G_1}\\
      &\lesssim h\enorm{e}_{2,h,G_1}+h\enorm{e}_{2,h,G_2}+\|e\|_{-1,G_2}\\
      &\lesssim\cdots\lesssim h\enorm{e}_{2,h,G}+\|e\|_{-p,G}\lesssim h\enorm e_{2,h,G}+\|e\|_{-s,q,G}.
    \end{split}
  \end{equation}
 The proof is completed.
\end{proof}

After obtaining the interior estimate Lemma \ref{lem:localerror}, we can proceed with the interior error estimates in energy norm.
For this purpose, we first consider a special case: $B_h(u_h,v_h)=0$, for any $ v_h\in S_h^0(\Omega_1)$, specifically, we state the following lemma.
\begin{lemma}\label{lem:homest}
Let $u_h\in S_h(\Omega_1)$ satisfies
\begin{equation}\label{eq:homest}
B_h(u_h,v_h)=0 \quad \forall v_h\in S_h^0(\Omega_1),
\end{equation}
and let $G_0\subset\subset G\subset\subset \Omega_1$ be concentric spheres. Then for $h$ small enough, we have
\begin{equation}\label{ineq:homest}
  \enorm{u_h}_{2,h,G_0}^\sim\lesssim h\enorm{u_h}_{2,h,G}^\sim+\|u_h\|_{-s,q,G}.
\end{equation}
\end{lemma}
\begin{proof}
For any continuous function $\phi$, the Galerkin projection $\tilde P_h\phi$ of $\phi$ onto $S_h^0(G)$ is determined by
\[\tilde B_h(\tilde P_h\phi,\chi)=\tilde B_h(\phi,\chi)\quad\forall \chi\in S_h^0(G).\]
Thanks to Lax-Milgram Theorem, we know $\tilde P_h\phi$ is well defined.
Let $G_0\subset\subset G'\subset\subset G''\subset\subset G$ be concentric spheres with \[\mathrm{dist}(G_0,\partial G')=\mathrm{dist}(G',\partial G'')=\mathrm{dist}(G'',\partial G)\geq\kappa h.\]
We set cut-off function $\omega\in C_0^\infty(G')$ with $\omega\equiv1$ on $G_0$. We denote $\tilde u_h=\omega u_h$, then the
triangle inequality and the piecewise Poincar\'e-Friedrichs inequality \cite{BWZ2004} imply that
\begin{equation}\label{ineq:homest1}
 \enorm{u_h}_{2,h,G_0}^\sim\leq \enorm{\tilde u_h}_{2,h,G}^\sim
 \lesssim \enorm{\tilde u_h-\tilde P_h\tilde u_h}_{2,h,G}^\sim+\enorm{\tilde P_h\tilde u_h}_{2,h,G}.
\end{equation}
To handle the first term in the righthand of above inequality, we assert by \eqref{ineq:app1} that
\begin{equation}\label{ineq:homest2}
 \enorm{\tilde u_h-\tilde P_h\tilde u_h}_{2,h,G}^\sim\lesssim \inf_{\xi\in S_h^0(G)}\enorm{\tilde u_h-\xi}_{2,h,G}^\sim\lesssim h\enorm{u_h}_{2,h,G}^\sim.
\end{equation}
Next we treat the $\enorm{\tilde P_h\tilde u_h}_{2,h,G}$ in \eqref{ineq:homest1}. For any $\eta\in S_h^0(G)$, let $\tilde \eta=\omega\eta$, we compute the difference of
$B_h(\tilde u_h,\eta)$ and $B_h(u_h,\tilde\eta)$. Utilizing the product rule, we have
\begin{equation}\label{eq:omegauh}
\begin{split}
  B_h(\tilde u_h,\eta)=&\sum_{T\in\mathcal T_h}\int_{T\cap G}(\omega D^2u_h+2\nabla\omega\otimes\nabla u_h+u_hD^2\omega):D^2\eta+\\
  &\sum_{E\in\mathcal E_h}\int_{E\cap G}\avg{\omega\frac{\partial^2u_h}{\partial n^2}+
  2\frac{\partial\omega}{\partial n}\frac{\partial u_h}{\partial n}+u_h\frac{\partial^2\omega}{\partial n^2}}\jp{\frac{\partial\eta}{\partial n}}+\\
  &\sum_{E\in\mathcal E_h}\int_{E\cap G}\left(\omega\avg{\frac{\partial^2\eta}{\partial n^2}}\jp{\frac{\partial u_h}{\partial n}}+\frac{\gamma}{|E|}
  \jp{\frac{\partial u_h}{\partial n}}\jp{\frac{\partial \eta}{\partial n}}\right)
  \end{split}
\end{equation}
and
\begin{equation}\label{eq:omegaeta}
\begin{split}
  B_h(u_h,\tilde \eta)=&\sum_{T\in\mathcal T_h}\int_{T\cap G}(\omega D^2\eta+2\nabla\omega\otimes\nabla \eta+\eta D^2\omega):D^2u_h+\\
  &\sum_{E\in\mathcal E_h}\int_{E\cap G}\avg{\omega\frac{\partial^2\eta}{\partial n^2}+
  2\frac{\partial\omega}{\partial n}\frac{\partial \eta}{\partial n}+\eta\frac{\partial^2\omega}{\partial n^2}}\jp{\frac{\partial u_h}{\partial n}}+\\
  &\sum_{E\in\mathcal E_h}\int_{E\cap G}\left(\omega\avg{\frac{\partial^2u_h}{\partial n^2}}\jp{\frac{\partial \eta}{\partial n}}+\frac{\gamma}{|E|}
  \jp{\frac{\partial u_h}{\partial n}}\jp{\frac{\partial \eta}{\partial n}}\right).
  \end{split}
\end{equation}
We apply the Green's formula to $\int_{T\cap G}(2\nabla\omega\otimes\nabla \eta+\eta D^2\omega):D^2u_h$ and use the magic formula $[\![ab]\!]=\avg a[\![ b]\!]+ \avg b[\![ a]\!]$ to deduce
\begin{equation}\label{eq:differnece}
\begin{split}
&B_h(\tilde u_h,\eta)- B_h(u_h,\tilde \eta)\\
=&\sum_{T\in\mathcal T_h}\int_{T\cap G}(2\nabla\omega\otimes\nabla u_h+u_hD^2\omega):D^2\eta-\mathrm{div}(2\nabla\omega\otimes\nabla \eta+\eta D^2\omega)\cdot\nabla u_h\\
&+2\sum_{E\in\mathcal E_h}\int_{E\cap G}\avg{2\frac{\partial\omega}{\partial n}\frac{\partial u_h}{\partial n}+u_h\frac{\partial^2\omega}{\partial n^2}}
\jp{\frac{\partial \eta}{\partial n}}\\
\lesssim& \left(\|u_h\|_{1,G'}^2+\sum_{E\in\mathcal E_h}|E|\left\|\avg{2\frac{\partial\omega}{\partial n}
\frac{\partial u_h}{\partial n}+u_h\frac{\partial^2\omega}{\partial n^2}}\right\|_{0,E\cap G}^2\right)^{\frac12}\enorm{\eta}_{2,h,G},
  \end{split}
\end{equation}
where we have used the fact $\omega\in C_0^\infty(G')$. Similarly,
using the trace inequality and the inverse inequality, we have
\begin{equation}\label{ineq:homest3}
 \sum_{E\in\mathcal E_h}|E|\left\|\avg{2\frac{\partial\omega}{\partial n}
\frac{\partial u_h}{\partial n}+u_h\frac{\partial^2\omega}{\partial n^2}}\right\|_{0,E\cap G}^2\lesssim \|u_h\|_{1,G''}^2.
\end{equation}
Therefore, we have by \eqref{eq:differnece} and \eqref{ineq:homest3} that
\begin{equation}\label{ineq:homest4}
B_h(\tilde u_h,\eta)\lesssim B_h(u_h,\tilde \eta)+\|u_h\|_{1,G''}\enorm{\eta}_{2,h,G}.
\end{equation}
Using \eqref{eq:homest} and \eqref{ineq:app1}, we obtain
\begin{equation}\label{ineq:homest5}
B_h(u_h,\tilde \eta)=\inf_{\chi\in S_h^0(G)}B_h(u_h,\tilde \eta-\chi)\lesssim h\enorm{u_h}_{2,h,G}\enorm{\eta}_{2,h,G}^\sim.
\end{equation}
Combing \eqref{ineq:homest4} \eqref{ineq:homest5} and taking $\eta=\tilde P_h \tilde u_h/\enorm{\tilde P_h \tilde u_h}_{2,h,G}^\sim$, then Lemma \ref{lem:localerror} leads to
\begin{equation}\label{ineq:homest6}
\enorm{\tilde P_h\tilde u_h}_{2,h,G}\lesssim B_h(\tilde u_h,\eta)\lesssim h\enorm{u_h}_{2,h,G}+\|u_h\|_{-s,q,G}.
\end{equation}
The desired result follows from \eqref{ineq:homest1}, \eqref{ineq:homest2} and \eqref{ineq:homest6}.
\end{proof}
\begin{lemma}\label{lem:locallemma}
Under the same condition of Lemma \ref{lem:homest}, we have
\begin{equation}\label{ineq:locallemma}
  \enorm{u_h}_{2,h,G_0}^\sim\lesssim \|u_h\|_{-s,q,G}.
\end{equation}
\end{lemma}
\begin{proof}

We choose integral $p$ such that  $W^{-s,q}(G)\hookrightarrow H^{-p}(G)$, let $G_0\subset\subset G_1\subset\subset\cdots\subset\subset G_{p+2}=G$
be concentric spheres with the distance of any two adjacent spheres is greater than $\kappa h$. We utilize Lemma \ref{lem:homest} $p+1$ times and
the inverse inequality to obtain
\begin{align*}
  \enorm{u_h}_{2,h,G_0}^\sim &\lesssim h\enorm{u_h}_{2,h,G_1}^\sim+\|u_h\|_{-s,q,G_1} \\
  &\lesssim\cdots\lesssim h^{p+1} \enorm{u_h}_{2,h,G_{p+1}}^\sim+\|u_h\|_{-s,q,G_{p+1}}\\
  &\lesssim \|u_h\|_{-p, G_{p+2}}+\|u_h\|_{-s,q,G_{p+1}}\lesssim \|u_h\|_{-s,q,G}.
\end{align*}
This ends the proof.
\end{proof}

Now we can state  the following interior energy norm estimate.
\begin{lemma}\label{lem:locinest}
 For two concentric spheres  $G_0\subset\subset G$, suppose $u\in H^l(\Omega)$,  where here and in what follows $l$ indicates
 an integer with $3\leq l\leq k+1$. Then
    \begin{equation}\label{ineq:locinest}
    \enorm{e}_{2,h,G_0}^\sim\lesssim h^{l-2}\|u\|_{l,G}+\|e\|_{-s,q,G},
  \end{equation}
  for some $s \geq 0$ and $q \geq 1$.
\end{lemma}
\begin{proof}
Let $G_0\subset\subset G_0'\subset\subset G'\subset\subset G\subset\subset \Omega$ be concentric spheres with
\[\mathrm{dist}(\partial G_0, \partial G_0')=\mathrm{dist}(\partial G_0', \partial G')=\mathrm{dist}(\partial G', \partial G)\geq\kappa h.\]
Let the cut-off function $\omega\in C_0^\infty(G')$, $\omega\equiv1$ on $G_0'$, and set $\tilde u=\omega u$.  By the triangle inequality,
     \begin{equation}\label{ineq:locinest1}
    \enorm{e}_{2,h,G_0}\leq \enorm{u-P_h\tilde u}_{2,h,G_0}+\enorm{P_h\tilde u-u_h}_{2,h,G_0},
  \end{equation}
  where the Galerkin projection $P_h\tilde u\in S_h^0(G)$ of $\tilde u$ is defined by
  \[B_h(P_h\tilde u,\chi)=B_h(\tilde u,\chi)\quad \forall \chi\in S_h^0(G).\]
By the standard approximation theory,
       \begin{equation}\label{ineq:locinest2}
 \enorm{u-P_h\tilde u}_{2,h,G_0}\leq \enorm{\tilde u-P_h\tilde u}_{2,h,G}\lesssim \inf_{\chi\in S_h^0(G)}\enorm{\tilde u-\chi}_{2,h,G}\lesssim h^{l-2}\|u\|_{l,G}.
  \end{equation}
 For the second term in the righthand of  \eqref{ineq:locinest1}, we observe that
 \[B_h(P_h\tilde u-u_h, \chi)=B_h(P_h\tilde u-u,\chi)-B_h(u_h-u,\chi)=0\quad\forall \chi\in S_h^0(G_0'),\]
  then Lemma \ref{lem:locallemma} leads to
         \begin{equation}\label{ineq:locinest3}
 \enorm{P_h\tilde u-u_h}_{2,h,G_0}\lesssim \|P_h\tilde u-u_h\|_{-s,q,G_0'}\lesssim \|e\|_{-s,q,G_0'}+\enorm{\tilde u-P_h\tilde u}_{2,h,G}.
  \end{equation}
  Plugging \eqref{ineq:locinest2} and \eqref{ineq:locinest3} into \eqref{ineq:locinest1}, we conclude the desired assertion.
\end{proof}
Based on the above Lemma, Lemma \ref{lem:dualone} and the same discussion as in the proof of Lemma \ref{lem:localerror}, we have the following lemma about the interior error estimate in $L^2$ norm.
\begin{lemma}\label{lem:locinestl2}
Under the same condition of Lemma \ref{lem:locinest}, we have estimate
\begin{equation}\label{ineq:locinestl2}
\|e\|_{0,G_0}\lesssim h^{\min\{l,l+k-3\}}\|u\|_{l,G}+\|e\|_{-s,q,G}.
\end{equation}
\end{lemma}

By the standard cover argument and Lemma \ref{lem:locinest}, Lemma \ref{lem:locinestl2}, we obtain the following local error estimate, which is the fundamental result in this subsection.
\begin{theorem}\label{thm:locinest}
   Let $\Omega_0\subset\subset\Omega_1\subset\subset \Omega$ be separated by $d=\mathcal O(1)$, then
  \begin{equation}\label{ineq:main1locinest}
    \enorm{e}_{2,h,\Omega_0}\lesssim h^{l-2}\|u\|_{l,\Omega_1}+\|e\|_{-s,q,\Omega_1},
  \end{equation}
  and
    \begin{equation}\label{ineq:main1locinestl2}
   \|e\|_{0,\Omega_0}\lesssim h^{\min\{l,l+k-3\}}\|u\|_{l,\Omega_1}+\|e\|_{-s,q,\Omega_1}.
  \end{equation}
  for some $s \geq 0$ and $q \geq 1$.
\end{theorem}
For our analysis in the forthcoming subsection about interior maximum estimates, we employ Lemma \ref{lem:locinest} to derive a scaling result and end this subsection.
\begin{lemma}\label{lem:scaling}
Let $\Omega_0\subset\subset \Omega_1\subset\subset \Omega$ with $d=\dist(\Omega_0,\partial \Omega_1)\geq \kappa h$. Then
\begin{equation}\label{ineq:scaling}
  \enorm{u-u_h}_{2,h,\Omega_0}\lesssim h^{l-2}\|u\|_{l,\Omega_1}+d^{-3}\|u-u_h\|_{0,1,\Omega_1}.
\end{equation}
\end{lemma}
\begin{proof}
To show the lemma, we shall show \eqref{ineq:scaling} with $\Omega_0=G_0$ and $\Omega_1=G_1$ being the spheres of radii $d/2$ and $d$ with same center $x_0=0$, respectively, then the claim follows from the covering argument.
We transform the domains $G_0$ and $G_1$ to the domains $\hat G_0$ and $\hat G_1$ respectively by the  variable substitution $\hat x=x/d$, then we have $\dist(\hat G_0,\partial\hat G_1)=\frac12$.
To proceed our analysis, we set $\hat u(\hat x)=u(xd)$ and $\hat u_h=u_h(xd)$ and denote $\hat S_h^0(\hat G_1)$ be the transferred space of $S_h^0(G_1)$, therefore, we have
 \begin{equation}\label{ineq:scaling-1}
   \hat B_h(\hat u-\hat u_h,\hat\chi)=0\quad \forall\hat\chi\in \hat S_h^0(\hat G_1),
 \end{equation}
where the transferred bilinear form $\hat B_h(\cdot,\cdot)$ is defined by:
\begin{align*}
\hat B_h(\hat v, \hat w) = &\sum_{\hat T\in\hat{\mathcal{T}}_{h}} \int_{\hat T}\hat D^2\hat v:\hat D^2\hat w d\hat x
	+ \sum_{\hat E\in\hat{\mathcal{E}}_{h}}\int_{\hat E} \avg{\frac{\hat \partial^2 \hat v}{\hat\partial \hat n^2}}\jp{\frac{\hat\partial\hat w}{\hat\partial \hat n}} d\hat s+\\
	&\sum_{\hat E\in\hat{\mathcal{E}}_{h}}\int_{\hat E} \jp{\frac{\hat\partial\hat v}{\hat\partial\hat n}}
	\avg{\frac{\hat\partial^2 \hat w}{\hat\partial \hat n^2}}d\hat s + \sum_{\hat E\in\hat{\mathcal{E}}_{h}}\frac{\gamma}{|\hat E|}\int_{\hat E}\jp{\frac{\hat\partial\hat v}{\hat\partial\hat n}}
	\jp{\frac{\hat\partial\hat w}{\hat\partial\hat n}} d\hat s.
\end{align*}
Therefore, Lemma \ref{lem:locinest} leads to
 \begin{equation}\label{ineq:scaling-2}
 \begin{aligned}
   \enorm{\hat u-\hat u_h}_{2,h,\hat G_0}&\lesssim (h/d)^{l-2}\|\hat u\|_{l,\hat G_1}+\|\hat u-\hat u_h\|_{0,1,\hat G_1}\\
   &\lesssim h^{l-2}d\|u\|_{l, G_1}+d^{-2}\| u- u_h\|_{0,1, G_1}.
   \end{aligned}
 \end{equation}
 Consequently, we conclude the result by employing the standard scaling estimate
 \begin{equation}\label{ineq:scaling-3}
    \enorm{u- u_h}_{2,h, G_0}\lesssim d^{-1} \enorm{\hat u-\hat u_h}_{2, h,\hat G_0}
 \end{equation}
 and \eqref{ineq:scaling-2}. The proof is completed.
\end{proof}

\subsection{Interior Maximum Norm Estimates}
\label{sec:localmaxestimate}
In this subsection, we study the interior maximum norm estimates of numerical solution for C0IP method. To this end, we first consider it with respect to the local coercivity bilinear
form $\tilde B_h(\cdot,\cdot)$, the main result is Theorem \ref{lem:localBtilde}.

The subsequent analysis hinges significantly upon the pivotal lemma presented herein, we will prove it in Appendix.
\begin{lemma}\label{lem:5.3}
Let $G_h\subset\subset\frac14G$ be concentric spheres and the radius of $G_h$ is $\mathcal O(h)$, $\varphi\in C_0^\infty(G_h)$, and let $v\in H^2(G)$, $v_h\in S_h(G)$ satisfy
\begin{equation}\label{eq:5.3-1}
  \tilde B(v,\psi)=(\varphi,\psi)\quad\forall \psi\in H^2(G),
\end{equation}
and
\begin{equation}\label{eq:5.3-2}
 \tilde B_h(v-v_h,\chi)=0\quad\forall \chi\in S_h(G),
\end{equation}
respectively. Then
\begin{equation}\label{ineq:lem5.3}
  \enorm{v-v_h}_{2,1,h,G}^\sim\lesssim h^3\cdot\Lambda(h)\|\varphi\|_{0,G_h},
\end{equation}
where
\[
\Lambda (h)=
\begin{cases}
  h^{-3/2}, & k=2 \\
  \ln(1/h), & k=3 \\
  1, & k\geq4.
\end{cases}
\]
\end{lemma}
Subsequently, the lemma aforementioned is employed to establish two local estimation lemmas.
\begin{lemma}\label{lem:5.1}
Let $\tilde u$ have compact support in the sphere $\frac12G$, and $\tilde u_h\in S_h(G)$ satisfies
\begin{equation}\label{ineq:5.1-cond}
  \tilde B_h(\tilde u-\tilde u_h, v_h)=0\quad \forall v_h\in S_h(G),
\end{equation}
then we have the following estimate
\begin{equation}\label{ineq:5.1}
  \|\tilde u-\tilde u_h\|_{0,\infty,\frac14G}\lesssim \|\tilde u\|_{0,\infty,G}+h^2\cdot\Lambda(h)\enorm {\tilde u}_{2,h,\infty,G}^\sim.
\end{equation}
\end{lemma}
\begin{proof}
  Let $\|\tilde u-\tilde u_h\|_{0,\infty,\frac14G}=|(\tilde u-\tilde u_h)(x_1)|$, for simplicity
in notation we shall assume that $x_1$ is the center of $G$. Let $G_h$ be a sphere  centered at $x_1$ with radius $R_{G_h}=\mathcal O(h)$. Then we derive by the inverse inequality and the
triangle inequality that
\begin{equation}\label{ineq:5.1-1}
\begin{aligned}
  \|\tilde u-\tilde u_h\|_{0,\infty,\frac14G}&\leq \|\tilde u\|_{0,\infty,G_h}+\|\tilde u_h\|_{0,\infty,G_h}\\
  &\lesssim \|\tilde u\|_{0,\infty,G_h}+h^{-1}\|\tilde u_h-\tilde u\|_{0,G_h}+h^{-1}\|\tilde u\|_{0,G_h}\\
  &\lesssim \|\tilde u\|_{0,\infty,G_h}+h^{-1}\|\tilde u_h-\tilde u\|_{0,G_h}.
  \end{aligned}
\end{equation}
Next we focus on the estimate of $\|\tilde u_h-\tilde u\|_{0,G_h}$ in the righthand of above inequality. To this end, for any given $\varphi\in C_0^\infty(G_h)$, we define $v\in H^2(G)$ and $v_h\in S_h(G)$ by
 \eqref{eq:5.3-1} and \eqref{eq:5.3-2}, respectively. Then \eqref{ineq:5.1-cond} gives that
\begin{equation}\label{ineq:5.1-2}
  (\tilde u-\tilde u_h,\varphi)=\tilde B_h(\tilde u-\tilde u_h,v)=\tilde B_h(\tilde u,v-v_h).
\end{equation}
Applying the Cauchy-Schwarz inequality and Lemma \ref{lem:5.3} to obtain
\begin{equation}\label{ineq:5.1-3}
  \tilde B_h(\tilde u,v-v_h)\lesssim \enorm{\tilde u}_{2,\infty,h,G}^\sim\enorm{v-v_h}_{2,1,h,G}^\sim\lesssim h^3\cdot\Lambda(h)\enorm{\tilde u}_{2,\infty,h,G}^\sim\|\varphi\|_{0,G_h}.
\end{equation}
As a consequence, combing \eqref{ineq:5.1-2} with \eqref{ineq:5.1-3}, we get
\begin{equation}\label{ineq:5.1-4}
\|\tilde u_h-\tilde u\|_{0,G_h}=\sup_{\varphi\in C_0^\infty(G_h)\setminus\{0\}}\frac{(\tilde u-\tilde u_h,\varphi)}{\|\varphi\|_{0,G_h}}\lesssim h^3\cdot\Lambda(h)\enorm{\tilde u}_{2,\infty,h,G}^\sim.
\end{equation}
The claim then follows from \eqref{ineq:5.1-1} and \eqref{ineq:5.1-4}.
\end{proof}

\begin{lemma}\label{lem:5.2}
Let $w_h\in S_h(G)$ satisfies
\begin{equation}\label{ineq:5.2--1}
  \tilde B_h(w_h, \chi)=0\quad\forall \chi\in S_h^0(G).
\end{equation}
We denote the center of $G$ by $x_0$, then it holds
\begin{equation}\label{ineq:5.2}
  |w_h(x_0)|\lesssim \|w_h\|_{-s,q,G}.
\end{equation}
\end{lemma}
\begin{proof}
For the substantiation of this lemma, it is imperative to rely on Lemma \ref{lem:5.2-lem1} concerning superapproximation and Lemma \ref{lem:green} addressing the estimation of Green's function. Specifically,
using Lemma \ref{lem:5.2-lem1}, there exists a function $\eta_h\in S_h^0(\frac34G)$ such that $\eta_h\equiv w_h$ on $\frac12G$ and $\enorm{\eta_h}_{2,h,\frac34G}^\sim\lesssim \enorm{w_h}_{2,h,G}^\sim$.
Let $G_h\subset\subset G$ be concentric and the radius of $G_h$ is $\mathcal O(h)$.
For any given function $\varphi\in C_0^\infty(G_h)$, taking $v$ and $v_h$ satisfy \eqref{eq:5.3-1} and \eqref{eq:5.3-2} respectively, we have
\begin{equation}\label{ineq:5.2-1}
  (\eta_h,\varphi)=\tilde B_h(\eta_h,v_h).
\end{equation}
Let $\chi_h\in S_h^0(\frac12G)$ with $\chi_h\equiv v_h$ on $\frac14G$ be as $\xi$ in Lemma \ref{lem:5.2-lem1}. We use \eqref{ineq:ineq:5.2-lem1-1}, \eqref{ineq:5.2--1} and the fact
$\tilde B_h(\eta_h,\chi_h)=\tilde B_h(w_h,\chi_h)$ to derive that
\begin{equation}\label{ineq:5.2-2}
\begin{split}
  (\eta_h,\varphi)&=\tilde B_h(\eta_h,v_h-\chi_h)\lesssim \enorm{\eta_h}_{2,h,\frac{3G}{4}}^\sim\enorm{v_h}_{2,h,\frac{3G}{4}\setminus \frac G8}^\sim\\
  &\lesssim \enorm{w_h}_{2,h,G}^\sim\enorm{v_h}_{2,h,\frac{3G}{4}\setminus \frac G8}^\sim.
  \end{split}
\end{equation}
Since $v_h$ satisfies $\tilde B_h(\chi,v_h)=0$ for any $\chi\in S_h^0(G\setminus G_h)$, Lemma \ref{lem:locallemma} gives
\begin{equation}\label{ineq:5.2-3}
\enorm{v_h}_{2,h,\frac{3G}{4}\setminus \frac G8}^\sim\lesssim \|v_h\|_{0,1,G}.
\end{equation}
Using Lemma \ref{lem:5.3} and \eqref{ineq:green1}, we obtain by the triangle inequality that
\begin{equation}\label{ineq:5.2-4}
\begin{split}
\|v_h\|_{0,1,G}&\lesssim \enorm{v_h-v}_{2,1,h,G}^\sim+\|v\|_{2,1,G}\\
&\lesssim h^3\cdot\Lambda(h)\|\varphi\|_{0,G_h}+h\|\varphi\|_{0,G_h}
\lesssim h\|\varphi\|_{0,G_h},
\end{split}
\end{equation}
suppose that $h$ is small enough. Consequently, we observe from \eqref{ineq:5.2-2}, \eqref{ineq:5.2-3} and \eqref{ineq:5.2-4} that
\begin{equation}\label{ineq:5.2-5}
  |w_h(x_0)|=|\eta_h(x_0)|\lesssim h^{-1}\|\eta_h\|_{0,G_h}\leq h^{-1}\sup_{\varphi\in C_0^\infty(G_h)\setminus\{0\}}\frac{(\eta_h,\varphi)}{\|\varphi\|_{0,G_h}}
  \lesssim \enorm{w_h}_{2,h,G}^\sim.
\end{equation}
Replacing $G$ by $\frac12G$ in the above derivation, we complete the proof by employing Lemma \ref{lem:locallemma}.
\end{proof}

We now state the main theorem in this subsection.
\begin{theorem}\label{lem:localBtilde}
Let $\Omega_0\subset\subset\Omega_1\subset\subset \Omega$, moreover, let $u$ and $u_h\in S_h(\Omega_1)$ satisfies
\[\tilde B_h(u-u_h,\chi)=0\quad\forall\chi\in S_h^0(\Omega_1).\]
Then the following estimate is true
\begin{equation}\label{ineq:localBtilde}
\begin{split}
  \|u-u_h\|_{0,\infty,\Omega_0}\lesssim& \inf_{\chi\in S_h(\Omega_1)}\Big\{\|u-\chi\|_{0,\infty,\Omega_1}+
  h^2\cdot\Lambda(h)\enorm{ u-\chi}_{2,\infty,h,\Omega_1}^\star\Big\}\\
  &\quad+\|u-u_h\|_{-s,q,\Omega_1},
\end{split}
\end{equation}
where \[\enorm{v}_{2,\infty,h,\mathcal A}^\star:=\enorm{v}_{2,\infty,h,\mathcal A}^\sim+\|v\|_{1,\infty,\mathcal A}+\max_{e\in\mathcal E_h}\|v\|_{0,\infty,e\cap\mathcal A}+
\max_{e\in\mathcal E_h}\left\|\avg{\frac{\partial v}{\partial n}}\right\|_{0,\infty,e\cap \mathcal A}.\]
\end{theorem}

\begin{proof}
  Let $G\subset\subset\Omega_1$ be a sphere of radius $R$ with center at $x_0\in\bar\Omega_0$, and $x_0$ is such that $\|u-u_h\|_{0,\infty,\Omega_0}=|(u-u_h)(x_0)|$.
Let $\omega\in C_0^\infty(\frac12G)$ with $\omega\equiv1$ on $\frac14G$ and set $\tilde u=\omega u$. Taking $\tilde u_h\in S_h(G)$ satisfies
\begin{equation}\label{ineq:main}
  \tilde B_h(\tilde u-\tilde u_h,v_h)=0\quad\forall v_h\in S_h(G).
\end{equation}
Then, we deduce by Lemma \ref{lem:5.1} that
\begin{equation}\label{ineq:main1}
\begin{aligned}
  |\tilde u(x_0)-\tilde u_h(x_0)|&\lesssim \|\tilde u\|_{0,\infty,G}+h^2\cdot\Lambda(h)\enorm{\tilde u}_{2,h,\infty,G}^\sim\\
  &\lesssim \|u\|_{0,\infty,G}+h^2\cdot\Lambda(h)\enorm u_{2,h,\infty,G}^\star.
  \end{aligned}
\end{equation}
For any $w_h\in S_h^0(\frac14G)$, there holds
\[\tilde B_h(\tilde u_h-u_h,w_h)=\tilde B_h(u-u_h,w_h)=0.\]
We employ Lemma \ref{lem:5.2}, Lemma \ref{lem:5.1} and \eqref{ineq:main1} to obtain
\begin{equation}\label{ineq:main2}
\begin{aligned}
  |(\tilde u_h-u_h)(x_0)|&\lesssim \|\tilde u_h-u_h\|_{-s,q,\frac14G}\\
  &\lesssim \|\tilde u_h-\tilde u\|_{0,\infty,\frac14G}+\|u-u_h\|_{-s,q,\frac14G}\\
  &\lesssim \|u\|_{0,\infty,G}+h^2\cdot\Lambda(h)\enorm u_{2,h,\infty,G}^\star+\|u-u_h\|_{-s,q,\frac14G}.
  \end{aligned}
\end{equation}
Using \eqref{ineq:main1} and \eqref{ineq:main2}, the triangle inequality implies that
\begin{equation}\label{ineq:main3}
\|u-u_h\|_{0,\infty,\Omega_0}\lesssim \|u\|_{0,\infty,\Omega_1}+h^2\cdot\Lambda(h)\enorm u_{2,h,\infty,\Omega_1}^\star+\|u-u_h\|_{-s,q,\Omega_1}.
\end{equation}
We rewrite $u-u_h=(u-\chi)-(u_h-\chi)$ for any $\chi\in S_h(\Omega_1)$, then \eqref{ineq:main3} completes the proof.
\end{proof}

Now we are ready to generalize Theorem \ref{lem:localBtilde} to the bilinear form $B(\cdot,\cdot)$ that is not locally positive definite. From now until the end of this section, $u$ and $u_h\in S_h(\Omega_1)$ satisfy
\begin{equation}\label{eq:mainsc4}
B_h(u-u_h,\chi)=0\quad\forall\chi\in S_h^0(\Omega_1).
\end{equation}
We first state the following lemma.
\begin{lemma}\label{lem:A-2}
Suppose $G_0\subset\subset G\subset\subset \Omega$ be concentric spheres and suppose  $u \in W^3_\infty(G)$, then we have the following estimate
\begin{equation}\label{ineq:A2}
  \|u-u_h\|_{0,\infty,G_0}\lesssim \|u\|_{3,\infty,G}+\|u-u_h\|_{0,G}.
\end{equation}
\end{lemma}
\begin{proof}
For any $\chi\in S_h^0(G)$, we have $B_h(u-u_h,\chi)=0$, that is
\begin{equation}\label{eq:A2-1}
  \tilde B_h(u-u_h,\chi)=K(u-u_h,\chi).
\end{equation}
Let $\psi\in H^2(G)$ be the unique solution of
\begin{equation}\label{eq:A2-2}
  \tilde B(\psi,v)=K(u-u_h,v)\quad\forall v\in H^2(G),
\end{equation}
and $\psi_h\in S_h(G)$ satisfies $\tilde B_h(\psi-\psi_h,\chi)=0$ for any $\chi\in S_h(G)$. Consequently,
\begin{equation}\label{eq:A2-3}
  \tilde B_h(u-u_h-\psi_h,\chi)=0\quad\forall\chi\in S_h^0(G).
\end{equation}
Applying Theorem \ref{lem:localBtilde} to equation \eqref{eq:A2-3}, using basic approximation theory and Sobolev's lemma \cite{AR1975} to obtain
\begin{equation}\label{ineq:A2-1}
  \|u-u_h-\psi_h\|_{0,\infty,G_0}\lesssim \|u\|_{3,\infty,G}+\|u-u_h\|_{0,G}+\|\psi_h\|_{0,G}.
\end{equation}
Now we bound $\|\psi_h\|_{0,\infty,G_0}$. By Theorem \ref{lem:localBtilde}, interpolation error estimates and Sobolev embedding theorem, we claim that 
\begin{equation}\label{ineq:A2-2}
  \begin{split}
  \|\psi_h\|_{0,\infty,G_0}&\lesssim \|\psi-\psi_h\|_{0,\infty,G_0}+\|\psi\|_{0,\infty,G_0} \\
  & \lesssim \|\psi-I_h\psi\|_{0,\infty,G}+h^2\cdot\Lambda(h)\enorm{\psi-I_h\psi}_{2,h,\infty,G}^\star\\
  &\quad+\|\psi-\psi_h\|_{0,G}+\|\psi\|_{0,\infty,G_0} \\
  &\lesssim \|\psi\|_{4,G}+\|\psi-\psi_h\|_{0,G}.
  \end{split}
\end{equation}
It is obvious that
\begin{equation}\label{ineq:A2-3}
  \|\psi-\psi_h\|_{0,G}\lesssim  \|\psi\|_{3,G}.
\end{equation}
Whence combing \eqref{ineq:A2-2}, \eqref{ineq:A2-3} and the regularity estimate
$\|\psi\|_{4,G}\lesssim \|u-u_h\|_{0,G}$, we get
\begin{equation}\label{ineq:A2-4}
 \|\psi_h\|_{0,\infty,G_0}\lesssim \|u-u_h\|_{0,G}.
\end{equation}
Notice that $\|\psi_h\|_{0,G}\lesssim \enorm{\psi}_{2,h,G}^\sim\lesssim \|u-u_h\|_{0,G}$, then the lemma follows from \eqref{ineq:A2-1},
\eqref{ineq:A2-4} and the triangle inequality.
\end{proof}
From Lemma \ref{lem:A-2}, we have the following lemma.
\begin{lemma}\label{lem:A-1}
Under the same condition of Lemma \ref{lem:A-2}, we have
\begin{equation}\label{ineq:A1}
  \|u-u_h\|_{0,\infty,\frac14G}\lesssim \|u\|_{3,\infty,G}+\|u-u_h\|_{-s,q}.
\end{equation}
\end{lemma}
\begin{proof}
  According to Lemma \ref{lem:A-2}, there holds
  \begin{equation}\label{ineq:AA-1-1}
    \|u-u_h\|_{0,\infty,\frac14G}\lesssim \|u\|_{3,\infty,\frac12G}+\|u-u_h\|_{0,\frac12G}.
  \end{equation}
  By means of Lemma \ref{lem:locinestl2}, we obtain
    \begin{equation}\label{ineq:AA1-2}
    \|u-u_h\|_{0,\frac12G}\lesssim \|u\|_{3,\infty,G}+\|u-u_h\|_{-s,q,G}.
  \end{equation}
Plugging \eqref{ineq:AA1-2} into \eqref{ineq:AA-1-1}, we complete the proof.
\end{proof}
We are now in a perfect position to show our main result in this section.
\begin{theorem}\label{thm:local}
The following estimate is true
\begin{equation}\label{ineq:local}
\begin{split}
  \|u-u_h\|_{0,\infty,\Omega_0}\lesssim& \inf_{\chi\in S_h(\Omega_1)}\Big\{\|u-\chi\|_{0,\infty,\Omega_1}+ h^2\cdot\Lambda(h)\enorm{ u-\chi}_{2,\infty,h,\Omega_1}^\star\Big\}\\
  &\quad+\|u-u_h\|_{-s,q,\Omega_1}.
\end{split}
\end{equation}
\end{theorem}
\begin{proof}
Let $\omega\in C_0^\infty(\frac12G)$ with $\omega\equiv1$ on $\frac14G$, denote by $\tilde u=\omega u$ and by $\tilde u_h\in S_h(G)$ the solution of
\begin{equation}\label{eq:local1}
  \tilde B_h(\tilde u-\tilde u_h,\chi)=0\quad\forall \chi\in S_h(G).
\end{equation}
According to Lemma \ref{lem:5.1}, we obtain
\begin{equation}\label{ineq:local1}
  \|\tilde u-\tilde u_h\|_{0,\infty,\frac14G}\lesssim \|u\|_{0,\infty,G}+h^2\cdot\Lambda(h)\enorm u_{2,h,\infty,G}^\star.
\end{equation}
Next we have the equation from \eqref{eq:local1} and \eqref{eq:mainsc4} that
\begin{equation}\label{eq:local2}
  B_h(\tilde u_h- u_h,\chi)=-K(\tilde u_h- \tilde u,\chi)\quad \forall \chi\in S_h^0(\tfrac14G).
\end{equation}
Let $G' = c_oG$ with sufficiently small positive constant $c_o$ and let $\psi\in H_0^2(G')$ be such that
\begin{equation}\label{eq:local3}
  B(\psi,v)=-K(\tilde u_h-\tilde u,v)\quad\forall v\in H_0^2(G').
\end{equation}
Equations \eqref{eq:local2} and \eqref{eq:local3} lead to
\begin{equation}\label{eq:local4}
  B_h(\psi-(\tilde u_h- u_h),\chi)=0\quad\forall \chi\in S_h^0(G').
\end{equation}
We deduce from Lemma \ref{lem:A-1} and \eqref{ineq:local1} that
\begin{equation}\label{ineq:local2}
  \begin{split}
    & \|\psi-(\tilde u_h- u_h)\|_{0,\infty,\frac14G'}\\
   \lesssim & \|\psi\|_{3,\infty,G'}+\|\psi-(\tilde u_h-u_h)\|_{-s,q,G'}\\
    \lesssim &\|\psi\|_{3,\infty,G'}+\|u-u_h\|_{-s,q,\frac14G}+\|u\|_{0,\infty,G}+h^2\cdot\Lambda(h)\enorm u_{2,h,\infty,G}^\star.
  \end{split}
\end{equation}
Then by the triangle inequality, \eqref{ineq:local1} and \eqref{ineq:local2}, we have
\begin{equation}\label{ineq:local3}
  \|u-u_h\|_{0,\infty,\frac14G'}\lesssim \|\psi\|_{3,\infty,G'}+\|u-u_h\|_{-s,q,\frac14G}+\|u\|_{0,\infty,G}
  +h^2\cdot\Lambda(h)\enorm u_{2,h,\infty,G}^\star.
\end{equation}
By elliptic regularity Lemma \ref{lem:rglBtilde} and \eqref{ineq:local1}, it holds that
\begin{equation}\label{ineq:local4}
  \|\psi\|_{3,\infty,G'}\lesssim \|\tilde u-\tilde u_h\|_{0,\infty,G'}\lesssim \|u\|_{0,\infty,G}+h^2\cdot\Lambda(h)\enorm u_{2,h,\infty,G}^\star.
\end{equation}
Thus we finally get from \eqref{ineq:local3} and \eqref{ineq:local4} that
\begin{equation}\label{ineq:local5}
   \|u-u_h\|_{0,\infty,\frac14G'}\lesssim \|u\|_{0,\infty,G}+h^2\cdot\Lambda(h)\enorm u_{2,h,\infty,G}^\star+\|u-u_h\|_{-s,q,G}.
\end{equation}
The proof is completed by rephrasing $u-u_h=(u-\chi)-(u_h-\chi)$ for any $\chi\in S_h(\Omega_1)$ and using \eqref{ineq:local5}.
\end{proof}
As a direct consequence of Theorem \ref{thm:local}, we have the following theorem.
\begin{theorem}\label{thm:locinfty}
   Let $\Omega_0\subset\subset\Omega_1\subset\subset \Omega$ be separated by $d=\mathcal O(1)$, and $u\in W_\infty^l(\Omega_1)$, we have
  \begin{equation}\label{ineq:locinfty}
   \|u-u_h\|_{0,\infty,\Omega_0}\lesssim h^l\cdot\Lambda(h)\|u\|_{l,\infty,\Omega_1}+\|u-u_h\|_{-s,q,\Omega_1},
  \end{equation}
  for some $s \geq 0$ and $q \geq 1$.
\end{theorem}
\section{Superconvergence analysis}
\label{sec:error}
In this section, we shall establish the superconvergent result for the recovered Hessian matrix $H_hu_h$ on translation-invariant meshes for equation \eqref{equ:c0ip}. Our main theoretical analysis tool is superconvergence by difference quotient, as discussed in \cite{NS1974, Wa1995}. This is possible partially due to  the pointwise error estimates for the C0IP method for biharmonic equations, as shown in Theorems \ref{thm:locinfty} and \ref{thm:locinest}.

\begin{figure}[!h]
   \centering
  \subcaptionbox{\label{fig:quadratic}}
   {\includegraphics[width=0.47\textwidth]{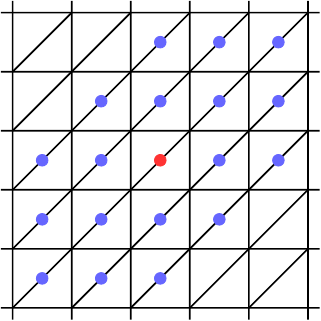}}
   \subcaptionbox{\label{fig:cubic}}
   {\includegraphics[width=0.47\textwidth]{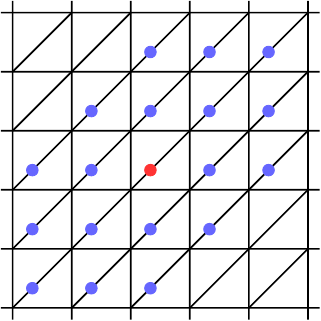}}
   \caption{Sampling point selection: (a) quadratic element;
    (b) cubic element.}\label{fig:sample}
\end{figure}

The key observation is that $H_h$ can be viewed as a finite difference operator. As explained in Remark \ref{rem:sample}, the selection  of sampling points in $\mathcal K_z$ is flexible. To adopt the argument of superconvergence by difference quotient, we chose the sampling points of the same type as the assembly point $z$. It's essential to emphasize that this restriction is solely for theoretical purposes and not for computational purposes. In Figure \ref{fig:sample}, we demonstrate two typical patterns in quadratic and cubic elements to illustrate how to choose sampling points of the same type in regularly patterned uniform meshes. Similar ideas are applicable to other translation-invariant meshes, such as chevron, criss-cross, union-Jack, and equilateral patterned uniform meshes. Then, we can express the recovered second-order derivative as
\begin{equation}\label{equ:diffq}
\left(H_h^{ab} u_h\right)(z)=\sum_{|\nu| \leq M} \sum_{i=1}^6 C_{\nu, h}^i u_h\left(z+\nu h \ell_i\right) = \sum_{|\nu| \leq M} \sum_{i=1}^6 C_{\nu, h}^iT_{\nu \tau}^{\ell_i} u_h(z),
\end{equation}
for $a, b\in \{x, y\}$, some nature number $M$,  and some direction $\ell_i$.

Based on the above observation, we can establish the following estimation for the Hessian recovery operator $H_h$.
\begin{lemma}\label{lem:hessian}
Let $\Omega_1 \subset \subset \Omega$ be separated by $d=\mathcal O(1)$; let the finite element space $S_h$ be translation invariant in the directions required by the Hessian recovery operator $H_h$ on $\Omega_1$; and let $u \in W_{\infty}^{k+1}(\Omega)$.  
 Then  on any interior region $\Omega_0 \subset\subset \Omega_1$, we have
\begin{equation}\label{equ:error}
	\|H_h(u-u_h)\|_{0, \infty, \Omega_0} \lesssim h^{k+1}\cdot\Lambda(h)\|u\|_{k+1,\infty,\Omega}+\|u-u_h\|_{-s,q,\Omega}
\end{equation}
for some $s \geq 0$ and $q \geq 1$.
\end{lemma}
\begin{proof}
	Note that  the bilinear form $B_h(\cdot,\cdot)$ are associated with constant coefficient fourth-order elliptic equation \eqref{equ:model}. The definition of translation invariant finite element space \eqref{equ:traninvariant} and the Galerkin orthogonality \eqref{equ:galorth} imply that
	\begin{equation}
	B_h(T_{\nu \tau}^{\ell}(u-u_h), v_h)=B_h(u-u_h, T_{-\nu \tau}^{\ell} v_h)=B_h(u-u_h,(T_{\nu \tau}^{\ell})^* v_h)=0,
\end{equation}
for any $v_h\in S_h^0(\Omega_1)$. From \eqref{equ:diffq}, we can deduce that
\begin{equation}
		B_h(H_h^{xx}(u-u_h), v_h)=B_h(u-u_h, (H_h^{xx})^* v_h)=0 \quad \forall v_h\in S_h^0(\Omega_1).
\end{equation}

Then, Theorem \ref{thm:local} implies that
\begin{equation}\label{equ:hessianerror}
\begin{split}
	&\|H_h^{xx}(u-u_h)\|_{0, \infty, \Omega_0} \\
\lesssim &  \inf_{v\in S_h(\Omega_1)}\Big\{\|H_h^{xx}u-v\|_{0,\infty,\Omega_1}+ h^2\cdot\Lambda(h)\enorm{H_h^{xx}u-v}_{2,\infty,h,\Omega_1}^\star\Big\}\\
&+\|H_h^{xx}(u-u_h)\|_{-s,q,\Omega_1}.
\end{split}
\end{equation}
The first term in the righthand of \eqref{equ:hessianerror} can be bounded by $h^{k+1}\cdot\Lambda(h)\|u\|_{k+1,\infty,\Omega_1}$ by using standard approximation theory. For the second term, we observe that
\begin{equation}\label{eq:Hxxdual}
  \|H_h^{xx}(u-u_h)\|_{-s,q,\Omega_1}=\sup_{\phi\in C_0^\infty(\Omega_1),\|\phi\|_{s,q',\Omega_1}=1}(H_h^{xx}(u-u_h),\phi)
\end{equation}
with $\frac1q+\frac{1}{q'}=1$. Let $\Omega_1+Mh$ be a subdomain  stretches out $Mh$ from $\Omega_1$,
 we use the fact that $\|(H_h^{xx})^*\phi\|_{0,1,\Omega_1+Mh}$ is bounded uniformly with respect to $h$ when $s\geq1$ to obtain
\begin{equation}\label{ineq:Hxx}
  \begin{split}
     (H_h^{xx}(u-u_h),\phi)&=(u-u_h,(H_h^{xx})^*\phi)\\
     &\lesssim \|u-u_h\|_{0,\infty,\Omega_1+Mh}\|(H_h^{xx})^*\phi\|_{0,1,\Omega_1+Mh}\\
     &\lesssim \|u-u_h\|_{0,\infty,\Omega_1+Mh}.
  \end{split}
\end{equation}
Applying Theorem \ref{thm:local} again, we have
\begin{equation}\label{ineq:localHxx}
\begin{split}
  \|u-u_h\|_{0,\infty,\Omega_1+Mh}\lesssim& \inf_{v\in S_h(\Omega_1)}\Big\{\|u-v\|_{0,\infty,\Omega}+ h^2\cdot\Lambda(h)\enorm{ u-v}_{2,\infty,h,\Omega}^\star\Big\}\\
  &\quad+\|u-u_h\|_{-s,q,\Omega}.
\end{split}
\end{equation}
Consequently, we combine \eqref{equ:hessianerror}-\eqref{ineq:localHxx} and  approximation theory to get
\[\|H_h^{xx}(u-u_h)\|_{0, \infty, \Omega_0}\lesssim h^{k+1}\cdot\Lambda(h)\|u\|_{k+1,\infty,\Omega}+\|u-u_h\|_{-s,q,\Omega}.\]
We can obtain similar error estimates for the $H_h^{xy}, H_h^{yx}$, and $H_h^{yy}$ by using the same argument. Thus, the error bound also holds when we replace $H_h^{xx}$ in $H_h$, which completes our proof.
\end{proof}

With the above preparation, we are now in the position to prove our main superconvergence results.
\begin{theorem}\label{equ:super}
	Let $\Omega_1 \subset \subset \Omega$ be separated by $d=\mathcal O(1)$; let the finite element space $S_h$ be translation invariant in the directions required by the Hessian recovery operator $H_h$ on $\Omega_1$; and let $u \in W_{\infty}^{k+2}(\Omega)$.
Then  on any interior region $\Omega_0 \subset\subset \Omega_1$,  we have for $k\geq3$ that
	\begin{equation}\label{equ:sup}
		\|Hu - H_hu_h\|_{0, \infty, \Omega_0} \lesssim h^{k}\|u\|_{k+2, \infty, \Omega}+\|u-u_h\|_{-s, q, \Omega},
	\end{equation}
for some $s \geq 0$ and $q \geq 1$.
\end{theorem}
\begin{proof}
 To establish the superconvergence of recovered Hessian matrix $H_hu_h$, we decompose the error as
 \begin{equation}\label{equ:decomp}
 	Hu - H_hu_h = (Hu  - H_hu_I) + (H_hu_I-H_hu) + (H_hu - H_hu_h).
 \end{equation}
 with $u_I=I_hu$ and we shall estimate errors term by term  in  \eqref{equ:decomp}. The first term can be estimated by using the
 consistency of $H_h$ as in Theorem \ref{thm:pp}, namely,
 \begin{equation}\label{equ:firstt}
 	\|Hu  - H_hu_I\|_{0, \infty, \Omega_0} \lesssim h^{k}\|u\|_{k+2, \infty, \Omega_0}.
 \end{equation}
 The second term can be bounded by the approximation theory as
 \begin{equation}\label{equ:secondt}
	\|H_hu_I-H_hu\|_{0,\infty,\Omega_0}\lesssim h^{k+1}\|u\|_{k+1,\infty,\Omega_0}.
\end{equation}
For the last term, we have
\begin{equation}\label{equ:thirdt}
	\|H_h\left(u-u_h\right)\|_{0, \infty, \Omega_0} \lesssim h^{k+1}\cdot\Lambda(h)\|u\|_{k+1, \infty, \Omega}+\|u-u_h\|_{-s, q, \Omega},
\end{equation}
 by Lemma \ref{lem:hessian}. Substituting \eqref{equ:firstt} - \eqref{equ:thirdt} into \eqref{equ:decomp} gives
 the desired results.
 \end{proof}
\begin{remark}\label{rmk:order}
	Compared with the superconvergence results of Hessian recovery for second-order elliptic equations, the observed superconvergence rate is $\mathcal{O}(h^{k})$ instead of $\mathcal{O}(h^{k+1})$. This limitation is due to the convergence rate of $Hu-H_hu_I$. In the specific case when $k=2$, we can only assert a superconvergence rate of $\mathcal{O}(h^{\frac32})$ based on Lemma \ref{lem:hessian}. However, numerical tests indicate that $H_hu_h$ exhibits superconvergence towards $D^2u$ in $L^\infty$ norm with a superconvergence rate of $\mathcal O(h^2).$
\end{remark}

Since the Hessian recovery operator is polynomial preserving, similar as Theorem \ref{thm:pp}, we have
    \begin{equation*}
      \|H u-H_hu_I\|_{0, \Omega}\lesssim h^{k}\|u\|_{k+2,\Omega}
    \end{equation*}
for any $u\in H^{k+2}(\Omega)$, see also section 2.4 in \cite{NZ2005} for a similar estimate for gradient recovery operator. Then using the interior
estimate in $L^2$ norm \eqref{ineq:main1locinestl2}  and an analogous argument of Theorem \ref{equ:super}, we have the following $\mathcal O(h^k)$ superconvergence result in $L^2$ norm,
which is also valid for $k=2$.
\begin{theorem}\label{equ:superL2}
	Let $\Omega_1 \subset \subset \Omega$ be separated by $d=\mathcal O(1)$; let the finite element space $S_h$ be translation invariant in the directions required by the Hessian recovery operator $H_h$ on $\Omega_1$; and let $u \in H^{k+2}(\Omega)$.
Then  on any interior region $\Omega_0 \subset\subset \Omega_1$, we have
	\begin{equation}\label{equ:supL2}
		\|Hu - H_hu_h\|_{0, \Omega_0} \lesssim h^{k}\|u\|_{k+2, \Omega}+\|u-u_h\|_{-s, q, \Omega}
	\end{equation}
for some $s \geq 0$ and $q \geq 1$.
\end{theorem}

	Although we can only prove the superconvergence of the recovered Hessian matrix on translation-invariant meshes, the numerical experiments in the next section show that this conclusion also holds for general unstructured meshes, including the Delaunay triangle meshes. Analogous to the recovery-type \textit{a posteriori} error estimator using the recovered gradient for second-order elliptic equations, we can construct a recovery-type \textit{a posteriori} error estimator using the recovered Hessian matrix for the fourth-order elliptic equation as:
\begin{equation}\label{equ:errind}
\eta_{h, T} = \|H_hu_h - D^2 u_h\|_{0, T}.
\end{equation}
Similarly, we can define the effectiveness index of the recovery-type \textit{a posteriori} error estimator as:
\begin{equation}
\kappa_h = \frac{\left(\sum_{T\in\mathcal{T}_h}\eta_{h,T}^2\right)^{1/2}}{\left(\sum_{T\in\mathcal T_h}\|D^2u - D^2 u_h\|_{0, T}^2\right)^{\frac12}}.
\end{equation}
The error estimator is considered asymptotically exact if $\lim\limits_{h\rightarrow 0}\kappa_h=1$.

\section{Numerical Experiments}
\label{sec:numer}
In this section, we present two numerical examples to support the theoretical results in Section \ref{sec:error}. The first investigation aims to demonstrate the superconvergence performance of the Hessian recovery method for C0IP discretization. The second one is designed to illustrate the asymptotic exactness of the recovery based {\it a posteriori} error estimator.

To demonstrate the superconvergence property, we split the set $\mathcal N_h$ into $\mathcal N_{1,h}$ and $\mathcal N_{2,h}$, where
\[\mathcal N_{1,h}:=\{z\in \mathcal N_h: \mathrm{dist}(z,\partial \Omega)\leq L\}\]
represents the set of nodes close to the boundary. Additionally, we introduce the boundary domain $\Omega_{1,h}$ as the domain associated with the union of elements whose vertices all lie in $\mathcal N_{1,h}$. The interior domain, denoted by $\Omega_{2,h}=\Omega\setminus \Omega_{1,h}$, is the complement of $\Omega_{1,h}$. For ease of notation, we define:
\begin{flalign*}
& (He)_0:=\Bigg(\sum_{T\in \mathcal T_h}\|D^2u-D^2u_h\|_{0,T}^2\Bigg)^{\frac12},\\
& (H^re)_0:=\Bigg(\sum_{T\subset\Omega_{2,h}}\|D^2u-H_hu_h\|_{0,T}^2\Bigg)^{\frac12},\\
& (H^re)_\infty:=\max_{T\subset\Omega_{2,h}}\|D^2u-H_hu_h\|_{0,\infty,T}.
\end{flalign*}

\subsection{Superconvergent results}
We consider the biharmonic equation $\Delta^2u=f$ on the domain $\Omega=(0,1)^2$. The true solution is given by
\[u=\sin^2(\pi x)\sin^2(\pi y),\]
where the source term and boundary condition are determined by the exact solution. The distance $L$ is taken as 0.1. We compute finite element errors on regular, Chevron, Criss-cross, Union-Jack patterns, as well as the Delaunay triangulation with regular refinement. In Tables \ref{table1}-\ref{table5}, we present the numerical results for the quadratic element.

From Tables \ref{table1}-\ref{table4}, we observe that $D^2u_h$ converges to $D^2u$ in the $L^2$ norm with a convergence rate of $\mathcal O(h)$, a well-known result in the literature. Meanwhile, the recovered Hessian $H_hu_h$ superconverges at the rate of $\mathcal O(h^2)$, as predicted by our Theorem \ref{equ:superL2}. Regarding the $L^\infty$ error, $H_hu_h$ converges at the rate of $\mathcal O(h^2)$ towards $D^2u$.

In Table 5, the errors for  the Delaunay triangulation meshes, an example of non-translation invariant meshes, are presented. Although our superconvergence theory is established on translation-invariant meshes, we observe the superconvergence phenomenon in these unstructured meshes.

Next, we report the convergence rates for the cubic and quartic elements, testing the convergence behavior on four types of translation-invariant meshes in the interior $L^\infty$ norm. From Fig. \ref{fig:sc_result}, we observe that for $k=3$, $H_hu_h$ superconverges to $D^2u$ at a rate of $\mathcal O(h^3)$ in $L^\infty$ norm, and for $k=4$, it superconverges to $D^2u$ at a rate of $\mathcal O(h^4)$. These results are consistent with our Theorem \ref{equ:super}.

\begin{table}[htb!]
\centering
\footnotesize
\caption{Numerical results for
the quadratic element on regular pattern uniform mesh}
\begin{tabular}{|r|c|c|c|c|c|c|}
\hline
 $1/h$ & $(He)_0$ & order& $(H^re)_0$ & order& $(H^re)_\infty$ & order\\ \hline\hline
 16 &1.94e+00&-&5.16e-01&-&8.38e-01&-\\ \hline
 32 &9.56e-01&1.02&1.22e-01&2.08&2.01e-01&2.06\\ \hline
 64 &4.75e-01&1.01&3.10e-02&1.98&4.96e-02&2.02\\ \hline
 128 &2.37e-01&1.00&7.83e-03&1.98&1.24e-02&2.00\\ \hline
 256 &1.19e-01&1.00&1.96e-03&2.00&3.10e-03&2.00\\ \hline
\end{tabular}
\label{table1}
\end{table}

\begin{table}[htb!]
\centering
\footnotesize
\caption{Numerical results for the quadratic element on Chevron pattern uniform mesh}
\begin{tabular}{|r|c|c|c|c|c|c|}
\hline
 $1/h$ & $(He)_0$ & order& $(H^re)_0$ & order& $(H^re)_\infty$ & order\\ \hline\hline
 16 &1.88e+00&-&3.67e-01&-&7.11e-01&-\\ \hline
 32 &9.30e-01&1.01&8.66e-02&2.09&1.77e-01&2.01\\ \hline
 64 &4.63e-01&1.00&2.19e-02&1.98&4.42e-02&2.00\\ \hline
 128 &2.32e-01&1.00&5.55e-03&1.98&1.11e-02&2.00\\ \hline
 256 &1.16e-01&1.00&1.39e-03&2.00&2.77e-03&2.00\\ \hline
\end{tabular}
\label{table2}
\end{table}

\begin{table}[htb!]
\centering
\footnotesize
\caption{Numerical results for the quadratic element on Criss-cross pattern uniform mesh}
\begin{tabular}{|r|c|c|c|c|c|c|}
\hline
 $1/h$ & $(He)_0$ & order& $(H^re)_0$ & order& $(H^re)_\infty$ & order\\ \hline\hline
 16 &1.26e+00&-&1.01e-01&-&1.94e-01&-\\ \hline
 32 &6.32e-01&1.00&2.54e-02&1.99&4.86e-02&2.00\\ \hline
 64 &3.16e-01&1.00&6.49e-03&1.97&1.22e-02&2.00\\ \hline
 128 &1.58e-01&1.00&1.65e-03&1.97&3.09e-03&1.98\\ \hline
 256 &7.90e-02&1.00&4.33e-04&1.93&8.36e-04&1.88\\ \hline
\end{tabular}
\label{table3}
\end{table}

\begin{table}[htb!]
\centering
\footnotesize
\caption{Numerical results for the quadratic element on Union-Jack pattern uniform mesh}
\begin{tabular}{|r|c|c|c|c|c|c|}
\hline
 $1/h$ & $(He)_0$ & order& $(H^re)_0$ & order& $(H^re)_\infty$ & order\\ \hline\hline
 16 &1.89e+00&-&2.50e-01&-&4.94e-01&-\\ \hline
 32 &9.44e-01&1.00&6.06e-02&2.04&1.20e-01&2.05\\ \hline
 64 &4.73e-01&1.00&1.55e-02&1.97&2.97e-02&2.01\\ \hline
 128 &2.36e-01&1.00&3.93e-03&1.98&7.50e-03&1.99\\ \hline
 256 &1.18e-01&1.00&1.00e-03&1.97&1.96e-03&1.94\\ \hline
\end{tabular}
\label{table4}
\end{table}

\begin{table}[htb!]
\centering
\footnotesize
\caption{Numerical results for the quadratic element Delaunay mesh with regular refinement}
\begin{tabular}{|r|c|c|c|c|c|c|}
\hline
 $1/h$ & $(He)_0$ & order& $(H^re)_0$ & order& $(H^re)_\infty$ & order\\ \hline\hline
 16 &2.35e+00&-&4.87e-01&-&1.10e+00&-\\ \hline
 32 &1.17e+00&1.01&1.17e-01&2.06&2.18e-01&2.34\\ \hline
 64 &5.82e-01&1.00&3.48e-02&1.74&9.36e-02&1.22\\ \hline
 128 &2.91e-01&1.00&1.03e-02&1.75&4.29e-02&1.13\\ \hline
 256 &1.45e-01&1.00&3.10e-03&1.74&2.00e-02&1.10\\ \hline
\end{tabular}
\label{table5}
\end{table}

\begin{figure}
   \centering
   \subcaptionbox{\label{fig:sc-rg}}
  {\includegraphics[width=0.49\textwidth]{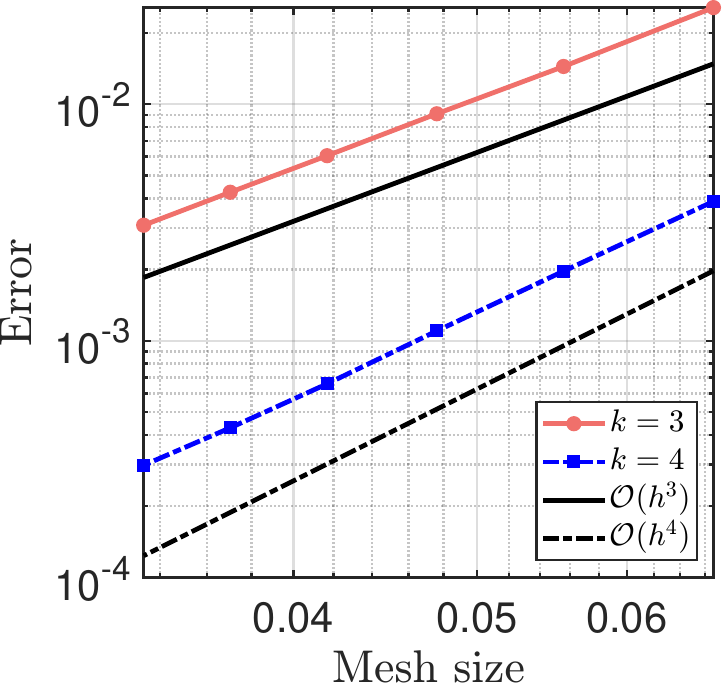}}
  \subcaptionbox{\label{fig:sc-ch}}
   {\includegraphics[width=0.485\textwidth]{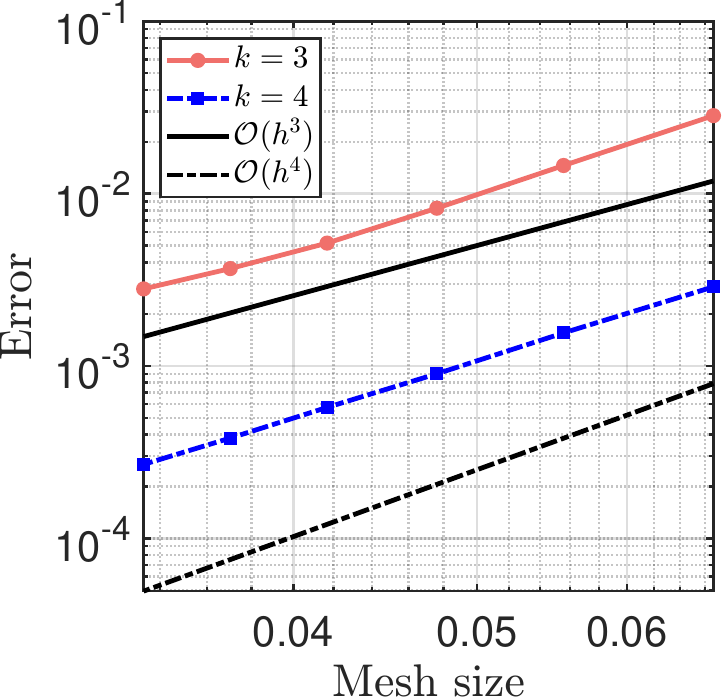}}
  \subcaptionbox{\label{fig:sc-cc}}
  {\includegraphics[width=0.49\textwidth]{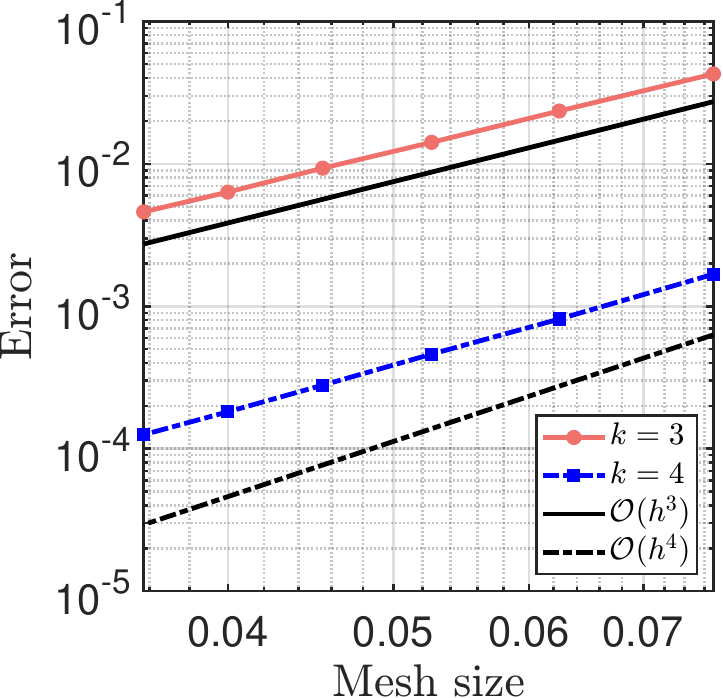}}
  \subcaptionbox{\label{fig:sc-uj}}
   {\includegraphics[width=0.485\textwidth]{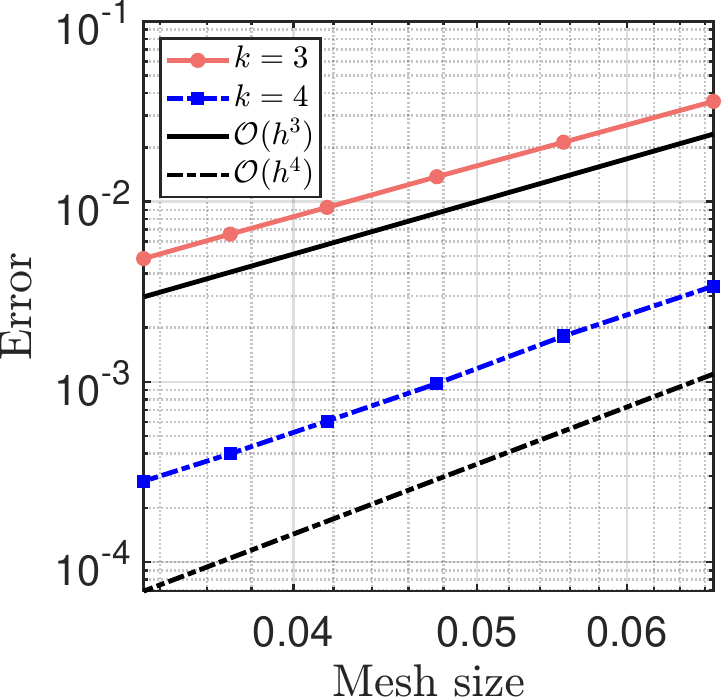}}

   \caption{The loglog plots of errors w.r.t the mesh sizes: (a)regular pattern; (b)Chevron pattern; (c) Criss-cross pattern; (d)Union-Jack pattern.
   }
   \label{fig:sc_result}
\end{figure}

\subsection{Adaptive C0IP methods}
In this test, we consider the biharmonic equation
\[ \Delta^2 u = 0\]
 on the L-shaped domain $\Omega = (-1,1)\times (-1,1)\setminus (0,1)\times (-1,0)$. The  exact solution in polar coordinate is  $u = r^{\frac{5}{3}}\sin(\frac53\theta)$. We see that the true solution $u$ has a singularity point at the origin. To resolve the singularity, our idea is to use the adaptive finite element method to solve this problem, employing the recovery-based {\it a posteriori} 
 error estimator defined in \eqref{equ:errind}.  The initial mesh and the mesh adaptively refined by
the Do\"{r}fler marking strategy with parameter 0.5 are shown in Fig.\ref{fig:adaptivemesh}.  As  depicted in Fig. \ref{fig:lshape_result}(a), an optimal convergence rate for the $H^2$ error can be observed, and we find that the recovered Hessian superconverges to the exact Hessian at a rate
 $\mathcal O(h^{1.3})$. In Fig. \ref{fig:lshape_result}(b), we plot the effectivity index $\kappa_h$ and observe that
$\kappa_h$ converges asymptotically to 1, demonstrating that the {\it a posteriori} error estimator \eqref{equ:errind} is asymptotically exact.
\begin{figure}
   \centering
   \subcaptionbox{\label{fig:initialmesh}}
  {\includegraphics[width=0.47\textwidth]{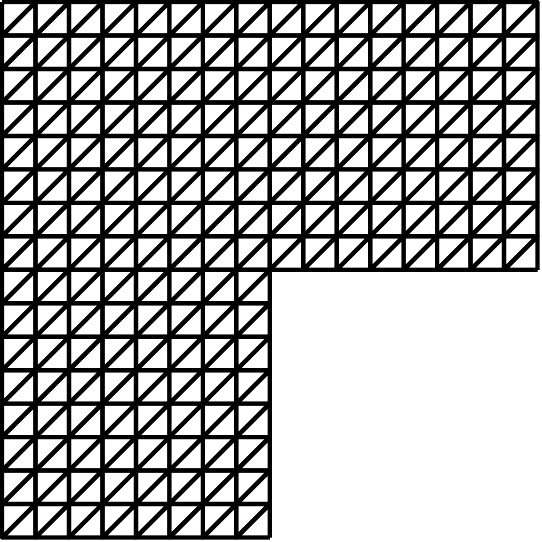}}
  \subcaptionbox{\label{fig:Gaussian_adaptive}}
   {\includegraphics[width=0.47\textwidth]{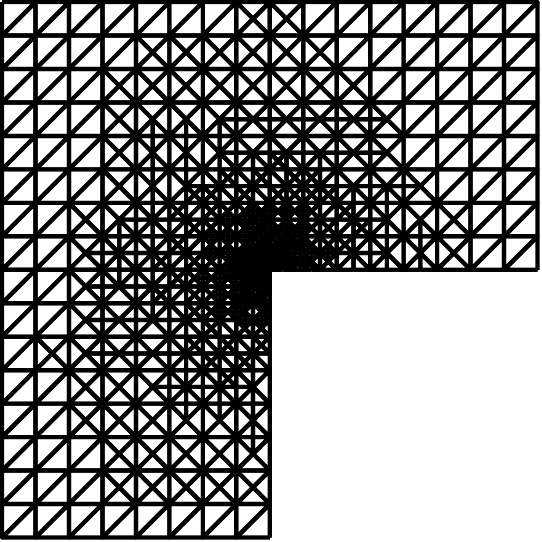}}
   \caption{Meshes for Lshape domain problem: (a) Initial mesh; (b) Adaptively refined mesh.}
   \label{fig:adaptivemesh}
\end{figure}

\begin{figure}
   \centering
   \subcaptionbox{\label{fig:gaussian_error}}
  {\includegraphics[width=0.485\textwidth]{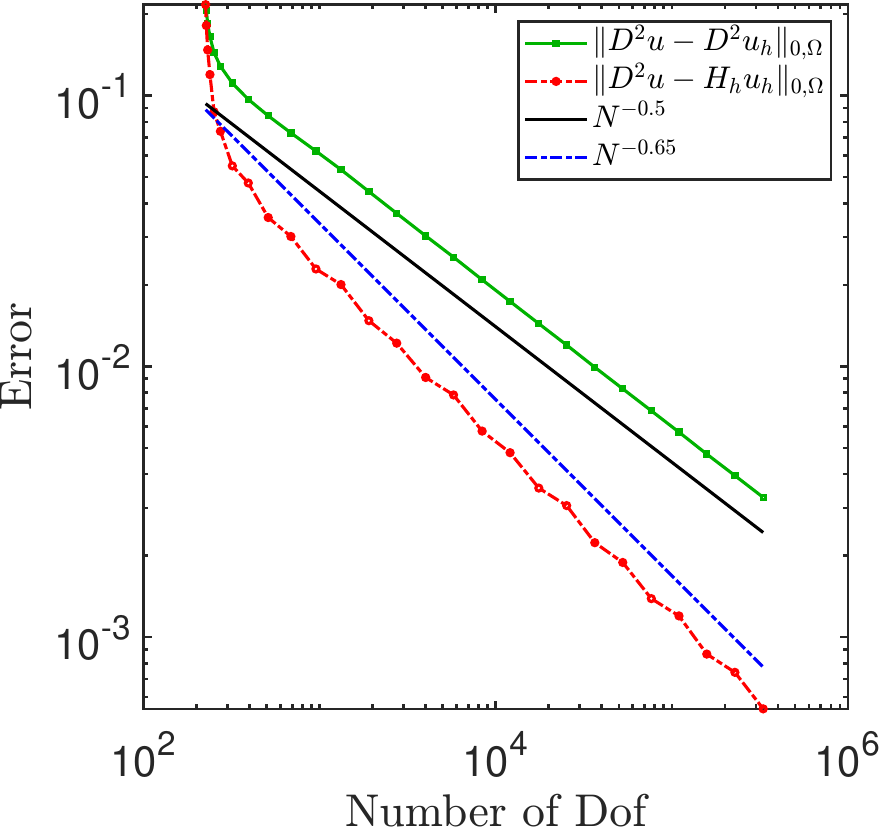}}
  \subcaptionbox{\label{fig:gaussian_ind}}
   {\includegraphics[width=0.485\textwidth]{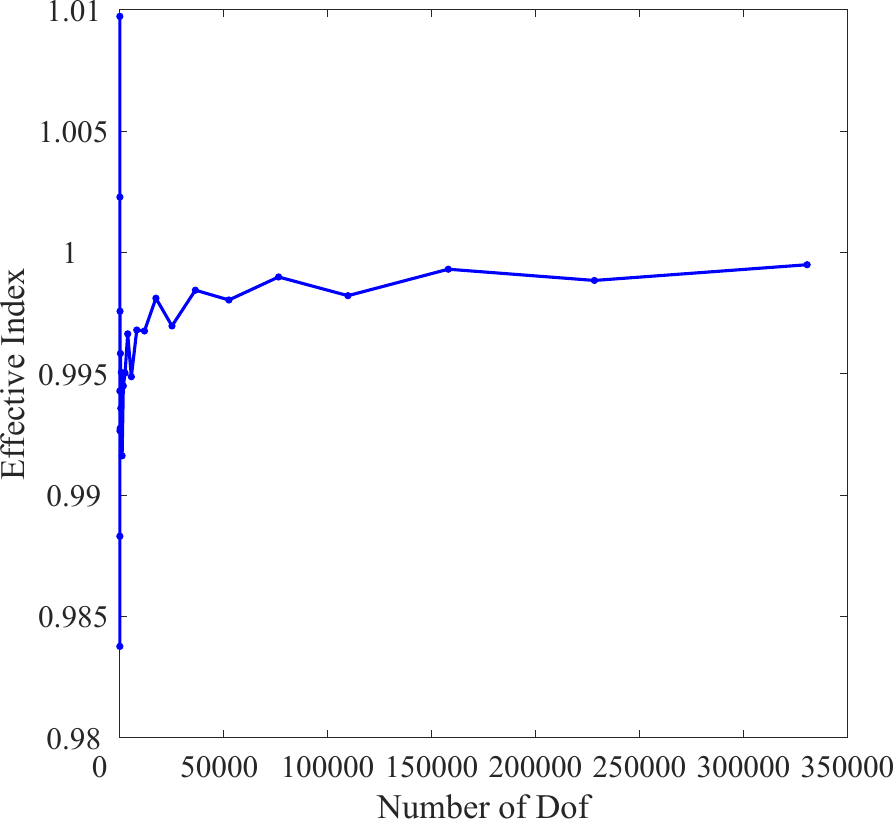}}
   \caption{Numerical result for test case 3: (a) Numerical errors; (b) Effective index.}
   \label{fig:lshape_result}
\end{figure}

\section{Conclusion}
\label{sec:con}
In this work, we conducted a superconvergent analysis of the Hessian recovery method for the C0IP  discretization of the biharmonic problem. The primary tools used for analyzing the superconvergence property are the established interior estimates theory and the application of the difference quotient on translation-invariant meshes. As a byproduct,  we developed an asymptotically exact {\it a posteriori} error  estimator.


\section*{Acknowledgment}
This work was supported in part by the Andrew Sisson Fund, Dyason Fellowship, the Faculty Science Researcher Development Grant of the University of Melbourne, and the NSFC grant 12131005.

\bibliographystyle{siamplain}
\bibliography{mybibfile}

\appendix

\section{Some basic estimates and the proof of Lemma \ref{lem:5.3}}

On the smooth domain, we recall the following regularity estimate, cf. \cite{GR1986}.
\begin{lemma}\label{lem:rgl}
Let $G\subset \Omega$ be a sphere, and $\varphi\in H^m(G)$, $m\geq-2$, then the unique solution $u\in H_0^2(G)$ to the equation
\begin{equation}\label{eq:rgl}
 B(u,v)=(\varphi,v)\quad\forall v\in H_0^2(G)
\end{equation}
has the regularity $u\in H_0^2(G)\cap H^{m+4}(G)$. Moreover, it holds the estimate
\begin{equation}\label{ineq:rgl}
  \|u\|_{m+4,G}\lesssim \|\varphi\|_{m,G}.
\end{equation}
\end{lemma}
Simultaneously, for the ``Neumann'' boundary condition, we assume the following regularity lemma holds.
\begin{lemma}\label{lem:rglBtilde}
Let $G\subset \Omega$ be a sphere, and $\varphi\in W^{p,q}(G)$, $p,q\geq0$, then $u\in H^2(G)$, the solution of the equation:
\begin{equation}\label{eq:rglBtilde}
\tilde B(u,v)=(\varphi,v)\quad\forall v\in H^2(G)
\end{equation}
has the regularity $u\in W^{p+4,q}(G)$ and it holds the estimate
\begin{equation}\label{ineq:rglBtilde}
  \|u\|_{p+4,q,G}\lesssim \|\varphi\|_{p,q,G}.
\end{equation}
Moreover, if $\varphi\in (H^1(G))^*$, then $u\in H^3(G)$ and $\|u\|_{3,G}\lesssim\|\varphi\|_{(H^1(\Omega))^*}$, where $(H^1(G))^*$ is the dual space of
$H^1(G)$.
\end{lemma}

We also collect some noteworthy superapproximation lemmas.
\begin{lemma}\label{lem:app}
Let $G_1\subset\subset G$ with $\dist(G_1,\partial G)\geq\kappa h$, and let $\omega\in C_0^\infty(G_1)$. Then for any $v\in S_h(G_2)$, when mesh size $h$ is small enough, there exists a $\xi\in S_h^0(G_1)$ such that
\begin{equation}\label{ineq:app1}
  \enorm{\omega v-\xi}_{2,h,G}^2\lesssim h^2\sum_{T\subset G}\|v\|_{2,T}^2.
\end{equation}
Moreover, let $G_4\subset\subset G_3\subset\subset G_2\subset\subset G_1$ with $\dist(G_4,\partial G_3)\geq\kappa h$, $\dist(G_3,\partial G_2)\geq\kappa h$. Then if $\omega\equiv 1$ on
$G_2$, we have $\xi\equiv v$ on $G_3$ and
\begin{equation}\label{ineq:app2}
  \enorm{\omega v-\xi}_{2,h,G}^2\lesssim h^2\sum_{T\subset G\setminus G_4}\|v\|_{2,T}^2.   
\end{equation}
\end{lemma}
\begin{proof}
Let $\xi=I_h(\omega v_h)$, by the interpolation error estimate and the trace inequality, we have for $h$ small enough that
\begin{equation}\label{ineq:app1-1}
  \enorm{\omega v-\xi}_{2,h,G}^2\lesssim h^{2(k-1)}\sum_{T_i}\|\omega v\|_{k+1,T_i}^2,
\end{equation}
where $\text{supp}\,\omega\subset\cup_i T_i\subset G$. By the fact that $D^{k+1}v=0$, the Leibniz rule and the inverse inequality
\[\|v\|_{k,T_i}\lesssim h^{2-k}\|v\|_{2,T_i},\]
we know the inequality \eqref{ineq:app1} holds by utilizing  \eqref{ineq:app1-1}. For \eqref{ineq:app2}, since $\omega\equiv 1$ on $G_2$, $\xi\equiv v$ on $G_3$ is apparent.
Similar as \eqref{ineq:app1-1}, we arrive at
\begin{equation}\label{ineq:app2-1}
  \enorm{\omega v-\xi}_{2,h,G}^2\lesssim h^{2(k-1)}\sum_{T_i'}\|\omega v\|_{k+1,T_i'}^2,
\end{equation}
with $\text{supp}\,\omega\setminus G_4\subset\cup_i T_i'\subset G$, then \eqref{ineq:app2} follows and which completes the proof.
\end{proof}
\begin{lemma}\label{lem:5.2-lem1}
Let $G_3\subset\subset G_2\subset\subset G_1\subset\subset G\subset\Omega$ be separated by $\kappa h$, and
suppose that $h$ is small enough. then for each $v\in S_h(G)$ there exists a $\xi\in S_h^0(G_1)$ with $\xi\equiv v$ on $G_2$ and
\begin{equation}\label{ineq:5.2-lem1}
  \enorm{v-\xi}_{2,h,G}^\sim\lesssim \enorm v_{2,h,G\setminus G_3}^\sim.
\end{equation}
\end{lemma}
\begin{proof}
Let $G_2'$ and $G_2''$ satisfy $G_2\subset\subset G_2'\subset\subset G_2''\subset\subset G_1$ and set $\omega\in C_0^\infty(G_2'')$ with $\omega\equiv1$ on $G_2'$.
It follows from Lemma \ref{lem:app} that, there exists a $\xi\in S_h^0(G_1)$ with $\xi\equiv v$ on $G_2$ satisfies
\begin{equation}\label{ineq:ineq:5.2-lem1-1}
  \enorm{\omega v-\xi}_{2,h,G}^\sim\lesssim h\enorm v_{2,h,G_1\setminus G_3}^\sim.
\end{equation}
We conclude the result by using $\enorm{(1-\omega)v}_{2,h,G}^\sim\lesssim \enorm v_{2,h,G\setminus G_3}^\sim$ and the triangle inequality.
\end{proof}

In the next lemma, we state the error estimates for C0IP discretization of the problem \eqref{eq:rglBtilde}.
\begin{lemma}\label{lem:c0ip}
Suppose $\varphi\in H^m(G)$, $m\geq0$. Let $u$ be the solution to the problem \eqref{eq:rglBtilde}, and $u_h\in S_h(G)$ be such that
\begin{equation}\label{eq:coip}
  \tilde B_h(u_h,v_h)=(\varphi,v_h)\quad \forall v_h\in S_h(G).
\end{equation}
Then we have
\begin{equation}\label{ineq:engnorm}
  \enorm{u-u_h}_{2,h,G}^\sim\lesssim h^{\min\{k-1,m+2\}}\|u\|_{m+4,G}.
\end{equation}
Furthermore, we have
\begin{equation}\label{ineq:L2}
  \|u-u_h\|_{0,G}\lesssim
  \begin{cases}
    h^2\|u\|_{3,G}, & k=2 \\
    h^{\min\{k+1,m+4\}}\|u\|_{m+4,G}, & k\geq 3.
  \end{cases}
\end{equation}
\end{lemma}
\begin{proof}
  The proof of this lemma is standard, we only provide a sketch. Using integration by parts, we have the consistent relation
  \begin{equation}\label{eq:consis}
    \tilde B_h(u-u_h,v_h)=0\quad\forall v_h\in S_h(G).
  \end{equation}
  According to the well known Sobolev extension Theorem, see \cite{AR1975}, there exists an operator $\mathfrak E: H^\ell(G)\to H^\ell(\mathbb R^2)$, $\ell\in \mathbb Z_+$, such that
  \[\mathfrak E v=v\quad\text{ on }G,\quad\text{and}\quad \|\mathfrak Ev\|_{\ell,\mathbb R^2}\lesssim \|v\|_{\ell,G}\]
  for any $v\in H^\ell(G)$. Consequently,  applying the coercivity of $\tilde B(\cdot,\cdot)$ and C\'ea Lemma, we have
  \begin{equation}\label{ineq:cea}
  \begin{split}
    \enorm{u-u_h}_{2,h,G}^\sim&\lesssim \enorm{u-I_hu}_{2,h,G}^\sim \lesssim   \enorm{\mathfrak E u-I_h(\mathfrak Eu)}_{2,h,\mathbb R^2}^\sim  \\
    &\lesssim h^{\min\{k-1,m+2\}}\|\mathfrak Eu\|_{m+4,\mathbb R^2}\lesssim h^{\min\{k-1,m+2\}}\|u\|_{m+4,G},
  \end{split}
  \end{equation}
 and \eqref{ineq:engnorm} is obtained. For the $L^2$ error estimate \eqref{ineq:L2},
based on \eqref{ineq:engnorm}, \eqref{ineq:cea} and consistency \eqref{eq:consis}, an application of the Aubin-Nitsche trick implies  \eqref{ineq:L2}.  The proof is completed.
\end{proof}

\begin{lemma}\label{lem:green}
Let $v\in H^2(G)$ and $\varphi\in L^1(G)$ satisfy
\begin{equation}\label{equ:green}
  \tilde B(v,w)=(\varphi,w)\quad\forall w\in H^2(G).
\end{equation}
 Then  we have
\begin{equation}\label{ineq:green1}
  \|v\|_{2,1,G}\lesssim \|\varphi\|_{0,1,G}.
\end{equation}
\end{lemma}
\begin{proof}
Due to the recent advance about the research on the Green's function of biharmonic operator
 on the ball by Karachik \cite{KVV2019,KVV2021}, we know
the dominated part of the Green's function $\mathrm G(x,y)$ of biharmonic equation is
\[\mathfrak G(x,y)=\frac{|x-y|^2}{4}(\ln|x-y|-1).\]
Therefore \[|D_x^\alpha\mathrm G(x,y)|\lesssim |\ln|x-y||\lesssim \frac{1}{|x-y|} \quad |\alpha|\leq2.\]
We assume this estimate is applicable for our case and note that
\begin{equation}\label{eq:green1-1}
D_x^\alpha v(x)=D_x^\alpha\int_G\mathrm G(x,y)\varphi(y)dy=\int_GD_x^\alpha\mathrm G(x,y)\varphi(y)dy,
\end{equation}
we arrive at
\begin{equation}\label{ineq:green1-1}
  \begin{split}
     \int_G|D_x^\alpha v(x)|dx&\leq\int_G\int_G| D_x^\alpha\mathrm G(x,y)\varphi(y)|dxdy\lesssim\int_G\int_G\frac{|\varphi(y)|}{|x-y|}dxdy\\
     &\lesssim \int_G|\varphi|dy \quad \forall |\alpha|\leq2.
  \end{split}
\end{equation}
The proof is completed.
\end{proof}

Finally, we can provide the proof of Lemma \ref{lem:5.3}.
\begin{proof}[Proof of Lemma \ref{lem:5.3}]
Let $\Omega_j$ denote the annuli
\[\Omega_j=\{x:2^{-j-1}<|x|<2^{-j}\},\]
and let $J$ be the largest integer such that $2^{-2J}\geq C_*h^{\bar k}$, with
\[\bar k=
\begin{cases}
           1, &  k=2 \\
           2, &  k\geq3,
         \end{cases}
\]
and $C_*$  is a large positive constant that will be specified later. Set $d_j=2^{-j}$, $\Omega_h=G\setminus\cup_{j=0}^J\Omega_j$, and let
\[\Omega_j^\ell=G\cap\bigcup_{i=j-\ell}^{j+\ell}\Omega_i,\quad \ell=1,2,\cdots.\]
Furthermore, we denote $e = v - v_h$. We have the basic relation
\begin{equation}\label{eq:5.3-3}
  \enorm{e}_{2,1,h,G}^\sim=\sum_{j=0}^J\enorm{e}^\sim_{2,1,h,\Omega_j}+\enorm{e}^\sim_{2,1,h,\Omega_h}.
\end{equation}
Using the Cauchy-Schwarz inequality and Lemma \ref{lem:c0ip} , it holds that
\begin{equation}\label{ineq:5.3-1}
  \enorm{e}_{2,1,h,\Omega_h}^\sim\leq CC_*^{\frac12}h^{\bar k/2}\enorm e^\sim_{2,h,\Omega_h}\leq CC_*^{\frac12}h^{\frac32\bar k}\|\varphi\|_{0,G_h}.
\end{equation}

Next we focus on the first term in \eqref{eq:5.3-3}. By Lemma \ref{lem:scaling}, we have
\begin{equation}\label{ineq:5.3-2}
  \enorm{e}^\sim_{2,1,h,\Omega_j}\leq Cd_j\enorm{e}^\sim_{2,h,\Omega_j}\leq Cd_j(h^{k-1}\|v\|_{k+1,\Omega_j^1}+d_j^{-3}\|e\|_{0,1,\Omega_j^1}).
\end{equation}

For the term $\|v\|_{k+1,\Omega_j^1}$ in the above inequality, we employ the estimate of the Green's function in the proof of Lemma \ref{lem:green} to obtain
that, for any $x\in \Omega_j^1$ and $3\leq|\alpha|\leq k+1$, we have
\begin{equation}\label{ineq:5.3-3}
\begin{split}
  |D^\alpha v(x)| & \leq\int_{G_h}|D_x^\alpha\mathrm G(x,y)||\varphi(y)|dy \\
   & \leq\left(\int_{G_h}|x-y|^{2(2-|\alpha|)}dy\right)^{\frac12}\|\varphi\|_{0,G_h}\lesssim h d_j^{1-k}\|\varphi\|_{0,G_h}.
\end{split}
\end{equation}
From \eqref{ineq:5.3-3}, we have
\begin{equation}\label{ineq:5.3-4}
\|v\|_{k+1,\Omega_j^1}\lesssim d_j^{2-k}h\|\varphi\|_{0,G_h}.
\end{equation}
Therefore \eqref{ineq:5.3-4} and \eqref{ineq:5.3-2} give that
\begin{equation}\label{ineq:5.3-5}
\begin{split}
  \sum_{j=0}^J\enorm{e}^\sim_{2,1,h,\Omega_j}&\lesssim \Bigg(\sum_{j=0}^Jd_j^{3-k}h^{k}\Bigg)\|\varphi\|_{0,G_h}+
  \sum_{j=0}^Jd_j^{-2}\|e\|_{0,1,\Omega_j^1}\\
  &\lesssim h^3\cdot\mathcal R(h)\|\varphi\|_{0,G_h}+\sum_{j=0}^Jd_j^{-2}\|e\|_{0,1,\Omega_j^1},
  \end{split}
\end{equation}
with
\[\mathcal R(h)=\begin{cases}
  h^{-1}, & k=2 \\
  \ln(1/h), & k=3 \\
  1, & k\geq4.
\end{cases}\]

Next we shall bound \[\Theta:=\sum_{j=0}^Jd_j^{-2}\|e\|_{0,1,\Omega_j^1}\]
in \eqref{ineq:5.3-5}.
Let \[\widetilde\Theta=\frac{1}{C_*h^{\bar k}}\|e\|_{0,1,\Omega_h}+\sum_{j=0}^Jd_j^{-2}\|e\|_{0,1,\Omega_j},\]
and inequality $\Theta\lesssim\widetilde\Theta$ holds. Now  Lemma \ref{lem:c0ip} gives
\begin{equation}\label{ineq:5.3-6}
\frac{1}{C_*h^{\bar k}}\|e\|_{0,1,\Omega_h}\leq Ch^{-\bar k/2}\|e\|_{0,\Omega_h}\leq Ch^{\frac32\bar k}\|\varphi\|_{0,G_h}.
\end{equation}

The main task below is to estimate the second term in $\widetilde\Theta$. For any $\eta\in C_0^\infty(\Omega_j)$, we consider the auxiliary problem: Find  $\psi\in H^2(G)$ such that
\begin{equation}\label{eq:5.3-*}
  \tilde B_h(\psi,w)=(\eta,w)\qquad \forall w\in H^2(G).
\end{equation}
We have for any $\chi\in S_{h}(G)$ that
\begin{equation}\label{ineq:5.3-7}
\tilde B_h(e,\psi-\chi)\lesssim \enorm{e}_{2,1,h,G\setminus\Omega_j^2}^\sim\enorm{\psi-\chi}_{2,\infty,h,G\setminus\Omega_j^2}^\sim+
\enorm{e}_{2,h,\Omega_j^2}^\sim\enorm{\psi-\chi}_{2,h,\Omega_j^2}^\sim.
\end{equation}
Taking $\chi=I_h\psi$, we obtain
\begin{equation}\label{ineq:5.3-8}
\enorm{\psi-\chi}_{2,\infty,h,G\setminus\Omega_j^2}^\sim \lesssim h^{\bar k}\|\psi\|_{\bar k+2,\infty,G\setminus \Omega_j^1}.
\end{equation}
For estimating $\|\psi\|_{\bar k+2,\infty,G\setminus \Omega_j^1}$ in the above inequality. Note that $\eta\in C_0^\infty(\Omega_j)$, then for any $x\in G\setminus \Omega_j^1$,
we derive by employing the Green's function of biharmonic operator that
\begin{equation}\label{ineq:5.3-9}
  \begin{split}
  |D^\alpha\psi(x)|&\leq\int_{\Omega_j}|D_x^\alpha\mathrm G(x,y)||\eta(y)|dy\\
  &\lesssim \|\eta\|_{0,\infty,\Omega_j}\int_{\Omega_j}\frac{1}{|x-y|^2}dy\lesssim \|\eta\|_{0,\infty,\Omega_j}\quad\forall|\alpha|\leq\bar k+2.
  \end{split}
\end{equation}
Thus $\|\psi\|_{\bar k+2,\infty,G\setminus \Omega_j^1}\lesssim \|\eta\|_{0,\infty,\Omega_j}$ and we further arrive at
\begin{equation}\label{ineq:5.3-10}
\enorm{\psi-\chi}_{2,\infty,h,G\setminus\Omega_j^2}^\sim \lesssim h^{\bar k}\|\eta\|_{0,\infty,\Omega_j}.
\end{equation}
Moreover, by Lemma \ref{lem:scaling}, we have
\begin{equation}\label{ineq:5.3-11}
  \enorm{e}_{2,h,\Omega_j^2}^\sim\lesssim h^{k-1}\|v\|_{k+1,\Omega_j^3}+d_j^{-3}\|e\|_{0,1,\Omega_j^3}.
\end{equation}
By the similar derivation of \eqref{ineq:5.3-4}, we have $\|v\|_{k+1,\Omega_j^3}\lesssim d_j^{2-k}h\|\varphi\|_{0,G_h}$, which leads to
\begin{equation}\label{ineq:5.3-12}
  \enorm{e}_{2,h,\Omega_j^2}^\sim\lesssim h^{k}d_j^{2-k}\|\varphi\|_{0,G_h}+d_j^{-3}\|e\|_{0,1,\Omega_j^3}.
\end{equation}
On the other hand, we have
\begin{equation}\label{ineq:5.3-13}
  \enorm{\psi-\chi}_{2,h,\Omega_j^2}^\sim\lesssim h^{\bar k}\|\psi\|_{\bar k+2,G}\lesssim h^{\bar k}\|\eta\|_{0,\Omega_j}\lesssim h^{\bar k}d_j\|\eta\|_{0,\infty,\Omega_j}.
\end{equation}
Therefore, we get from \eqref{ineq:5.3-12} and \eqref{ineq:5.3-13} that
\begin{equation}\label{ineq:5.3-14}
     \enorm{e}_{2,h,\Omega_j^2}^\sim \enorm{\psi-\chi}_{2,h,\Omega_j^2}^\sim \lesssim \Big(h^{k+\bar k}d_j^{3-k}\|\varphi\|_{0,G_h}+h^{\bar k}d_j^{-2}\|e\|_{0,1,\Omega_j^3}\Big)\|\eta\|_{0,\infty,\Omega_j}.
\end{equation}
Consequently, from \eqref{eq:5.3-*}, \eqref{ineq:5.3-7}, \eqref{ineq:5.3-10} and \eqref{ineq:5.3-14}, we arrive at:
\begin{equation}\label{ineq:5.3-15}
  \begin{split}
     \|e\|_{0,1,\Omega_j} & =\sup_{\eta\in C_0^\infty(\Omega_j)}\frac{\tilde B_h(e,\psi-\chi)}{\|\eta\|_{0,\infty,\Omega_j}} \\
       & \lesssim h^{\bar k}\enorm{e}_{2,1,h,G\setminus\Omega_j^2}^\sim+h^{k+\bar k}d_j^{3-k}\|\varphi\|_{0,G_h}+h^{\bar k}d_j^{-2}\|e\|_{0,1,\Omega_j^3}.
  \end{split}
\end{equation}
Now using the definition of $\widetilde\Theta$, estimates \eqref{ineq:5.3-6} and \eqref{ineq:5.3-15}, we have
\begin{equation}\label{ineq:5.3-16}
  \begin{aligned}
    \widetilde\Theta\leq&\, C\Bigg(h^{\frac32\bar k}+h^{k+\bar k}\sum_{j=0}^Jd_j^{1-k}\Bigg)\|\varphi\|_{0,G_h}\\
    &\quad+C\sum_{j=0}^J\frac{h^{\bar k}}{d_j^2}\Big(\enorm{e}_{2,1,h,G\setminus\Omega_j^2}^\sim+d_j^{-2}\|e\|_{0,1,\Omega_j^3}\Big).
  \end{aligned}
\end{equation}
Note that $\frac{h^{\bar k}}{d_j^2}\leq\frac1{C_*}$ and $\frac32\bar k\leq h^{k+\bar k}\sum_{j=0}^Jd_j^{1-k}$, inequality \eqref{ineq:5.3-16} leads us to
\begin{equation}\label{ineq:5.3-17}
    \widetilde\Theta\leq Ch^{\frac32\bar k}\|\varphi\|_{0,G_h}+\frac{C}{C_*}\enorm e_{2,1,h,G}^\sim+\frac{C_1}{C_*}\widetilde\Theta.
\end{equation}
Taking $C_*=2C_1$, there holds that
\begin{equation}\label{ineq:5.3-18}
    \widetilde\Theta\leq Ch^{\frac32\bar k}\|\varphi\|_{0,G_h}+\frac{C}{C_*}\enorm e_{2,1,h,G}^\sim.
\end{equation}
Therefore, we conclude from \eqref{eq:5.3-3}, \eqref{ineq:5.3-1}, \eqref{ineq:5.3-5}, \eqref{ineq:5.3-17} that 
\begin{equation}\label{ineq:5.3-19}
    \enorm e_{2,1,h,D}^\sim\leq C\Big(C_*^{\frac12}h^{\frac32\bar k}+h^3\cdot\mathcal R(h)\Big)\|\varphi\|_{0,G_h}+\frac{C_2}{C_*}\enorm e_{2,1,h,G}^\sim.
\end{equation}
Setting $\Lambda(h)=h^{\frac32\bar k}+h^3\cdot\mathcal R(h)$ and choosing $C_*=\max\{2C_1,2C_2\}$ in \eqref{ineq:5.3-19}, we finally obtain our desired result.
\end{proof}

\end{document}